\documentclass[a4paper, 11pt]{article}
\usepackage[papersize={210truemm,297truemm},margin=3.0truecm,nohead]{geometry}

\usepackage[pdftex]{graphicx, hyperref}          

\usepackage[dvipdfmx]{color}

\usepackage{amsmath}						
\usepackage{amsfonts}
\usepackage{latexsym}
\usepackage{amssymb}
\usepackage{bm}
\usepackage{ulem}
\usepackage{cases}
\usepackage{wrapfig,float}

\allowdisplaybreaks[3]

\usepackage{amsthm}							
\usepackage{ascmac}
\usepackage{mathtools}

\theoremstyle{definition}
\newtheorem{theorem}{Theorem}[section]
\newtheorem{prop}[theorem]{Proposition}
\newtheorem{lemma}[theorem]{Lemma}
\newtheorem{cor}[theorem]{Corollary}
\newtheorem{rem}[theorem]{Remark}
\newtheorem{definition}[theorem]{Definition}
\newtheorem{example}[theorem]{Example}
\newtheorem{setting}[theorem]{Setting}
\newtheorem{notation}[theorem]{Notation}

\newcounter{res}
\newtheorem{result}[res]{Main Result}

\numberwithin{equation}{section}

\newcommand{\1}{\mbox{1}\hspace{-0.25em}\mbox{l}}

\DeclareMathOperator*{\rs}{\widehat{\mathbb{C}}}
\DeclareMathOperator*{\nn}{\mathbb{N}}
\DeclareMathOperator*{\rat}{\mathrm{Rat}}
\DeclareMathOperator*{\poly}{\mathrm{Poly}}
\DeclareMathOperator*{\cm}{\mathrm{CM}}
\DeclareMathOperator*{\ocm}{\mathrm{OCM}}

\DeclareMathOperator{\supp}{supp}
\DeclareMathOperator{\diam}{diam}

\title{Non-i.i.d.\ random holomorphic dynamical systems
 and the probability of tending to infinity
\footnote{Date: 23 October 2018. 
Published  in  Nonlinearity 32 (2019) No. 10, 3742-3771. 
\  MSC: 37F10, 37H10. 
Keyword: random holomorphic dynamical systems, Markov random dynamical systems, randomness-induced phenomenon, noise-induced phenomenon, cooperation principle}
}
\date{\empty}

\author{Hiroki Sumi\\
Course of Mathematical Science, Department of Human Coexistence,\\
 Graduate School of Human and Environmental Studies, Kyoto University\\
  Yoshida Nihonmatsu-cho, Sakyo-ku, Kyoto, 606-8501, Japan\\
E-mail: sumi@math.h.kyoto-u.ac.jp\\
\href{http://www.math.h.kyoto-u.ac.jp/~sumi/index.html}{http://www.math.h.kyoto-u.ac.jp/$\sim$sumi/index.html}\\  \\
Takayuki Watanabe (corresponding author)\\
Course of Mathematical Science,  Department of Human Coexistence,\\
 Graduate School of Human and Environmental Studies, Kyoto University\\
  Yoshida Nihonmatsu-cho, Sakyo-ku, Kyoto, 606-8501, Japan\\
E-mail: watanabe.takayuki.43c@st.kyoto-u.ac.jp}

\begin{document}
\maketitle
\begin{abstract}
We consider random holomorphic dynamical systems on the Riemann sphere whose choices of maps are related to Markov chains. Our motivation is to generalize the facts which hold in i.i.d. random holomorphic dynamical systems.  In particular, we focus on the function $\mathbb{T}_{\infty, \tau}$ which represents the probability of tending to infinity. We show some sufficient conditions which make $\mathbb{T}_{\infty, \tau}$ continuous on the whole space and we characterize the Julia sets in terms of the function $\mathbb{T}_{\infty, \tau}$  under certain assumptions.  
\end{abstract}
\setcounter{subsection}{0}
\section{Introduction}
\subsection{Background}
We consider discrete-time random dynamical systems which are {\bf not i.i.d.}
The theory of random dynamics is rapidly growing both theoretically and experimentally. 
We focus in this paper on random {\bf holomorphic} dynamical systems on the Riemann sphere $\rs$ from the mathematical viewpoint. 
Using complex analysis, we can investigate the systems deeply.  

The first study of random holomorphic dynamics was given by Fornaess and Sibony \cite{fs91}.
They investigated  independent and identically-distributed (i.i.d.) random dynamical systems on $\rs$ constructed by small perturbations $\{ f_c \}_{c \in B(c_0, \delta)}$ of a rational map $f_{c_0} \colon \rs \to \rs$ which depend holomorphically on the parameter $c$, where $B(c_0, \delta)$ denotes the open ball with center $c_0$ and radius $\delta$ endowed with the normalized Lebesgue measure.  
They showed that, if $\delta$ is small and $f_{c_0}$ has $k \geq 1$ attractive cycles $\gamma_1, \dots, \gamma_k$,
then there exist $k$ continuous functions $T_{\gamma_j} \colon \rs \to [0,1]\, (j =1, \dots, k)$ with the following properties:
(i) $\sum_{j =1}^k  T_{\gamma_j}(z) =1$ for all $z \in \rs$, and
(ii) for all $z \in \rs$, the random orbit $f_{c_N}\circ  \cdots \circ f_{c_2}\circ f_{c_1} (z)$ tends to the attractive  basin of $\gamma_j$  as $N \to \infty$ with probability  $T_{\gamma_j} (z)$.
(See \cite[Theorem 0.1]{fs91}.)

In \cite{sumi11} the first author generalized this theorem to the case where noise is  not small and 
 deeply analyzed  the function $T_{A}$ which represents the probability of tending to an attracting minimal set $A$. 
These results are called the Cooperation Principles. 
His strategy is to consider both random dynamics of rational maps and dynamics of rational semigroups, which are semigroups of non-constant rational maps on $\rs$ where the semigroup operation is functional composition.
For details on  rational semigroups, see \cite{hm}, \cite{stank12}, \cite{sumi00}. 

The first author introduced the random relaxed Newton's method in \cite{sumi17} and suggested that the random relaxed Newton method might be a more useful method to compute the roots of polynomials than the classical deterministic Newton's method.
The key is that sufficiently large noise collapses bad attractors and makes the system more stable.

These works find new  phenomena which cannot hold in deterministic dynamics.
The phenomena are called {\bf noise-induced phenomena} or {\bf randomness-induced phenomena}, which are of great interest from the mathematical viewpoint. 
For more research on random holomorphic dynamical systems and related fields, see \cite{bb03}, \cite{com11}, \cite{js15}, \cite{js17}, \cite{msu}, \cite{mu18}, \cite{sumi00}, \cite{sumi11}, \cite{sumi17},  \cite{uz16}.

However, most of the previous studies concern i.i.d.\ random dynamical systems.
It is very natural  to generalize the settings and consider non-i.i.d random dynamical systems.
In this paper, we especially treat {\bf random dynamical systems with  ``Markovian rules'' whose noise depends on the past}.  

We extend the theory of i.i.d.\ random dynamical systems and {\bf we find new (noise-induced) phenomena which cannot hold in  i.i.d.\ random dynamical systems}. 
Moreover, our studies may be applied to the skew products whose base dynamical systems have Markov partitions.

We believe that this research will contribute not only toward mathematics but also toward applications to the real world.
One motivation for studying dynamical systems is to analyze mathematical models used in the natural or social sciences.  
Since the environment changes  randomly, 
it is natural to investigate random dynamical systems which describe the time evolution of systems with probabilistic terms.  
In this sense, it is quite important to understand ``Markovian'' noise because there are a lot of systems whose noise depends on the past.

Therefore the study of Markov random dynamical systems is natural and meaningful from both the pure and applied mathematical viewpoint. 
In this paper we aim to generalize the theory of i.i.d.\ random holomorphic dynamical systems and  the theory of rational semigroups simultaneously to the setting of random dynamical systems with Markovian rules and the associated set-valued dynamical systems. 

It is essentially new to consider the set-valued dynamical systems with Markovian rules itself, which we call graph directed Markov systems.
Although our concept  is similar to that of \cite{mu03}, these are completely different. 
In \cite{mu03} Mauldin and Urba\'nski are concerned with  the limit sets of systems of contracting maps, but in this paper we discuss the Julia sets of general continuous maps and clarify the relationship between the Julia sets of rational semigroups and that of graph directed Markov systems.

\subsection{Main results}

We now introduce our rigorous settings and present our main results.
Let $\rat$ be the space of non-constant holomorphic maps on $\rs$ and let $m \in \nn$.
We endow $\rat$ with the distance $\kappa$ defined by $\kappa ( f, g) := \sup_{z \in \rs} d(f(z), g(z))$ where $d$ denotes the spherical distance on $\rs$.
Suppose that $m^2$ Borel measures $ (\tau_{ij})_{i, j =1, \dots, m}$ on $\rat$ satisfy  $ \sum_{ j =1}^m \tau_{i j}(\rat) = 1$ for all $i =1, \dots, m$.
For the given  $\tau = (\tau_{ij})_{i, j =1, \dots, m}$, we consider the Markov chain on $\rs \times \{1, \dots, m\}$ whose transition probability from $(z, i ) \in \rs \times \{1, \dots, m\}$ to $ B \times \{ j \}$ is 
$$\mathbb{P} \left( (z,i), B \times \{j\} \right) = \tau_{ij}(\{ f \in \rat  ; f(z) \in B\}) $$
where $B$ is a Borel subset of $\rs $  and $ j \in  \{1, \dots, m\}$.
This system is called the {\it rational Markov random dynamical system} (rational MRDS for short) induced by $\tau$ .
For the rest of this subsection, we consider such systems.

Roughly speaking, the MRDS induced by $\tau = (\tau_{ij})_{i, j =1, \dots, m}$ describes the following random dynamical system on the phase space $\hat{\Bbb{C}}.$
Fix an initial point $z_{0} \in \rs$ and choose a vertex $i = 1, \dots, m$ (with some probability if we like). 
We choose a vertex $i_{1} = 1, \dots, m$ with probability $\tau_{i i_{1} }(\rat)$ and choose a map $f_{1}$ according to the probability distribution $ \tau_{i i_{1}} / \tau_{i i_{1} }(\rat).$ 
Repeating this, we randomly choose a vertex $i_{n}$ and a map $f_{n}$ for each $n$-th step. 
We in this paper investigate the asymptotic behavior of random orbits of the form $f_{n} \circ \dots \circ f_{2} \circ f_{1} (z_{0}).$ 

In  particular, we can apply functional analytical method by extending the phase space from $\rs$ to $\rs \times \{ 1, \dots, m \}$. 
More precisely, we consider iterations of a single transition operator, and it enables us  
to analyze the MRDS and the above random dynamical system deeply 
(see Section \ref{setting_sec}).  

Define the vertex set as $V:= \{ 1,2, \dots , m\} $ and the directed edge set as $$E :=\{ (i ,j) \in V \times V ; \, \tau _{ij}(\rat)  > 0 \}. $$
We set $S_{\tau} :=(V,E,(\supp \tau_e)_{e \in E})$ and utilize the terminology of directed graphs following \cite{mu03}.
We define $i : E \to V$ (resp.\ $t : E \to V$) as the projection to the  first (resp.\ second) coordinate and we call $i(e)$ (resp.\ $t(e)$) the initial (resp.\ terminal) vertex of $e \in E$.

A word  $e = (e_1,e_2, \dots, e_N) \in E^N$ with length $N \in \nn $ is said to be {\it admissible} if $t(e_n) = i(e_{n+1})$ for all  $n =1,2, \dots , N-1$. 
For this word $e$, we call $i(e_1)$ (resp. $t(e_N)$) the initial (resp. terminal) vertex of $e$ and we denote it by $i(e)$ (resp. $t(e)$). 
For each $i, j \in V$,  we define the following sets.
\begin{align*}
H_i^j(S) &:= \{f_{N} \circ  \dots  \circ f_{1};\, \exists N \in \nn, \exists e=(e_1, \dots, e_N) \in E^N \text{ s.t. }\\
 &     f_{n} \in \supp \tau_{e_n} (\forall n =1, \dots, N) \text{ and } e \text{ is admissible with } i(e)=i, t(e)=j\}, \\
	J_i (S_\tau)&:= \{ z \in \rs ;\bigcup_{j \in V} H_i^j(S_\tau) \text{ is not equicontinuous on any neighborhood of } z\}, \\
	J_{\ker,i} (S_\tau)&:=  \bigcap_{ j \in V :  H_i^j(S_\tau) \neq \emptyset} \bigcap_{h \in H_i^j(S) } h^{-1} (J_j(S)).
\end{align*}

The compact set $J_i (S_\tau)$ is called the Julia set at $i \in V$, which is the set of all initial points where the dynamical system sensitively depends on initial  conditions.
The subset $J_{\ker,i} (S_\tau)$ is called the kernel Julia set at $i \in V$.

To present our main results, we introduce the following Borel probability measures $\tilde{\tau}_i$ on $(\rat  \times E )^{\nn}$. 

\begin{definition}
We define Borel probability measures $\tilde{\tau}_i\, (i =1, \dots, m)$  on $(\rat  \times E )^{\nn}$ by 
 \begin{align*}
 &\tilde{\tau}_i \left(A'_1 \times \cdots \times A'_N \times \prod_{N+1}^{\infty} (\rat  \times E) \right) \\
 =& \begin{cases}
     \tau_{e_1}(A_1) \cdots \tau_{e_N}(A_N), & \text{if } (e_1, \dots , e_N)   \text{ is admissible  with }  i(e_1)=i \\
     0, & \text{otherwise}
   \end{cases}
   \end{align*}
for $N$ Borel sets $A_n \,(n=1,\dots,N)$ of  $\rat $ and for  $(e_1, \dots , e_N) \in E^N$ where
$A'_n = A_n \times \{ e_n \} $.
\end{definition}

For each element $\xi = (\gamma_n, e_n )_{n \in \nn}$ of $\supp \tilde{\tau}_{i}$ we can naturally consider the non-autonomous dynamics of $\xi$ and  define the Julia set $J_\xi$ as the set of non-equicontinuity of $\{ \gamma_N \circ \cdots \circ \gamma_1\}_{N \in \nn}$.
The following is a partial generalization of the Cooperation Principle.

\begin{result}[Proposition \ref{pathwisenull_prop}]\label{res-stable} 
If $J_{\ker,j} (S_\tau) = \emptyset$ for all $j \in V$,
then the ``averaged system'' is stable and the non-autonomous Julia set $J_\xi$ is of (Lebesgue) measure-zero for $\tilde{\tau}_i$-almost every $\xi$.
 \end{result}
We say that the system $S_\tau$ is {\it irreducible} if the directed graph $(V,  E)$ is strongly connected.
Using the theory of rational semigroups, we have the following result.

\begin{result}[Corollary \ref{cor-density}]\label{res-dense}
If $S_\tau$ is irreducible and $\# J_j(S_\tau) \geq 3$ for some $j =1, \dots,m$,  
then $$J_i(S_\tau) =\overline{ \bigcup_{ h \in H_i^i(S_\tau)} \{ \text{repelling fixed points of }h \} } = \overline{\bigcup_{\xi \in \supp \tilde{\tau}_i} J_{\xi}}$$
for all $i =1, \dots , m$.
Here $\# A$ denotes the cardinality of a set $A$.
\end{result}

We next focus on systems of polynomial maps on $\rs$ and the functions  which represent the probability of tending to infinity. 
For the rest of this subsection, suppose that $S_\tau$ is irreducible and $\supp \tau_e$ is a compact subset of the space $\poly$ of all polynomial maps on $\rs$  of degree $2$ or more for each $e \in E$. 

\begin{definition}\label{Tinfty_intro}
We define the function $\mathbb{T}_{\infty, \tau} \colon \rs \times V \to [0, 1]$ by
\[\mathbb{T}_{\infty ,\tau}(z,i) := \tilde{\tau}_i (\{ \xi = (\gamma_n , e_n )_{n \in \nn}  ; \, d(\gamma_{n}\circ \dots \circ \gamma_1 (z), \infty) \to 0 \, (n \to \infty) \} ) \]
for any point  $(z,i) \in \rs \times V$.
\end{definition}

We have the following results regarding the relation between the kernel Julia sets $ J_{\ker, i} (S_\tau)$ and the continuity of $\mathbb{T}_{\infty, \tau}$.

\begin{result}[Proposition \ref{prop-tInfty}]\label{res_Tconti}
	If $J_{\ker,j} (S_\tau) = \emptyset$ for some $j \in V$,
	then $\mathbb{T}_{\infty, \tau}$ is continuous on $\rs \times V$.
\end{result}

\begin{result}[Corollary \ref{cor-kernelJuliaCond} (ii)]\label{res_addnoise}
	Suppose that there exists $e \in E$ such that 
	$$\supp \tau_e \supset \{ f + c ; | c -c_0| < \varepsilon \}$$
	 for some $f \in \poly$, $c_0 \in \mathbb{C}$ and $\varepsilon >0$.
	Then  $J_{\ker,j} (S_\tau) = \emptyset$ for some $j \in V$ and hence $\mathbb{T}_{\infty, \tau}$ is continuous on $\rs \times V$.
\end{result}

Roughly speaking, if there are sufficiently many maps in one system, then the maps cooperate with one another and thereby  eliminate the chaos on average.
Consequently the function is continuous on the whole space.
This phenomenon cannot hold in deterministic dynamical systems since $\mathbb{T}_{\infty, \tau}$ takes the value $0$ on the filled-in Julia set  and the value $1$ outside of it.

Let us consider systems with  finite maps.
In this case, we need certain conditions which make $\mathbb{T}_{\infty, \tau}$ continuous.

\begin{definition}
We say that a system $S_\tau$  satisfies the {\it backward separating condition} if $f_{1} ^{-1} (J_{t(e_1)}(S) ) \cap f_{2} ^{-1} (J_{t(e_2)}(S) ) = \emptyset$ for every  $e_1, e_2 \in E$ with the same initial vertex and for every $f_{1}\in \supp \tau_{e_1}, f_{2} \in \supp \tau_{e_2} $, except the case $e_1 =e_2$ and $f_{1} =f_{2}  $.
\end{definition}

We say that $S_\tau$ is {\it essentially non-deterministic} if there exist $e_1, e_2 \in E$ with $i(e_1)=i(e_2)$ and there exist $f_1 \in \supp \tau_{e_1}, f_2 \in \supp \tau_{e_2}$ such that either $e_1 \neq e_2$ or $f_1 \neq f_2$.

We now present a result for  systems with finite maps regarding the continuity of $\mathbb{T}_{\infty, \tau}$ and the set of points where $\mathbb{T}_{\infty, \tau}$  is not locally constant.

\begin{result}[Lemma \ref{sepcond_lemma}, Proposition \ref{prop-intEmpty},  Theorem \ref{theorem-Tinfty}]\label{res-T}
	Suppose that the polynomial system $S_\tau$ satisfies the  backward separating condition.
If $\supp \tau_e$ is finite for each $e \in E$,
then $J_i(S_\tau)$ has no interior points for each $i \in V$ and we have  either $\mathbb{T}_{\infty, \tau} \equiv 1$ or 
$$J_i(S_\tau) = \{ z \in \mathbb{C} ; \mathbb{T}_{\infty, \tau}(\cdot, i) \text{ is not constant on any neighborhood of } z\}$$
for each $i \in V$.
Moreover, if additionally $S_\tau$ is essentially non-deterministic,
then $\mathbb{T}_{\infty, \tau}$ is continuous on $\rs \times V$.
\end{result}

The former part of Main Result \ref{res-T} is a generalization of the classical fact that the Julia set of polynomial $f$ of degree $2$ or more is the boundary of the filled-in Julia set of $f$.
However, the latter part of Main Result \ref{res-T} indicates a kind of randomness-induced phenomenon and is a generalization of the fact known in i.i.d. cases \cite[Lemma 3.75]{sumi11}.

We next present the applications of the main results.
For a fixed  $m \in \nn$, given $f_1, \dots,f_m \in \poly$ and a given irreducible stochastic matrix $P=(p_{ij})_{i,j = 1, \dots,m}$, we define  $\tau_{ij}$ as the  measure $p_{ij} \delta_{f_i}$, where $\delta_{f_i}$ denotes the Dirac measure at $f_i$.
We consider the polynomial MRDS induced by  $\tau = (\tau_{ij})$.
Let  $p =(p_1, \dots, p_m)$ be  the positive vector such that $\sum_{i=1}^m p_i =1$ and $p P =p$.
Set $T_{\infty, \tau} \colon \rs \to [0,1]$ as $T_{\infty, \tau} (z) := \sum_{i=1}^m p_i \mathbb{T}_{\infty, \tau}(z,i)$.
In other words, we consider the random dynamical system whose choice of maps is as follows: we choose a map $f_{i_1}$ with probability $p_{i_1}$ at the first step, and after choosing a map $f_{i_N}$ we choose the next map $f_{i_{N+1}}$ with probability $p_{i_N i_{N+1}}$ at the ($N+1$)-st step, where $i_1, \dots, i_N, i_{N+1} \in \{1, \dots, m\}$. 

\begin{cor}[Corollary \ref{cor-foeExample}]
Suppose that  ${T}_{\infty, \tau} \not \equiv 1$
	and $J_i (S_\tau)\cap J_j(S_\tau) = \emptyset$ for all  $i,j \in V$ with $i\neq j$.
	Then  $\bigcup _{i\in V}J_{i}(S_{\tau })$ has no interior points and 
	\begin{align*}
		 \bigcup_{i \in V} J_i(S_\tau)
		 = \{ z \in \mathbb{C} ; {T}_{\infty, \tau} \text{ is not constant on any neighborhood of } z\}.
	\end{align*}
Moreover, if there exist $i,j,k \in \{1, \dots , m \}$ with $j \neq k$  such that $p_{ij} >0$ and $p_{ik} >0$  in addition to the assumption above, 
then $T_{\infty,\tau}$ is continuous on $\rs$.
\end{cor}

There are new phenomena in Markov random dynamical systems which cannot hold in i.i.d.\ random dynamical systems.
More precisely, in the i.i.d.\ case, if the function ${T}_{\infty, \tau}$ is not identically $1$, then there exists $z_0 \in \mathbb{C}$ such that ${T}_{\infty, \tau}(z_0) =0 $ (see \cite{sumi11} or Lemma \ref{lem-filledEquiv}).
However, we show the following result. 
See also Figure \ref{fig-IntromrkvT} and Figure \ref{fig-section}. 

\begin{result}[Proposition \ref{prop-mrkv} and Example \ref{ex-mrkvT}]\label{res-6}
There exists $\tau=(\tau_{ij})_{i, j}$ with $\supp \tau_{ij} \subset \poly $ such that $T_{\infty, \tau}$ is continuous, ${T}_{\infty, \tau} \not \equiv 1$,  ${T}_{\infty, \tau}(z) > 0$ for each  $z \in \rs $ and $S_{\tau}$ is irreducible. 
\end{result}

 \vspace{-5mm}

\begin{figure}[htbp]
\begin{minipage}{71truemm}
\begin{center}
\includegraphics[width=4cm]{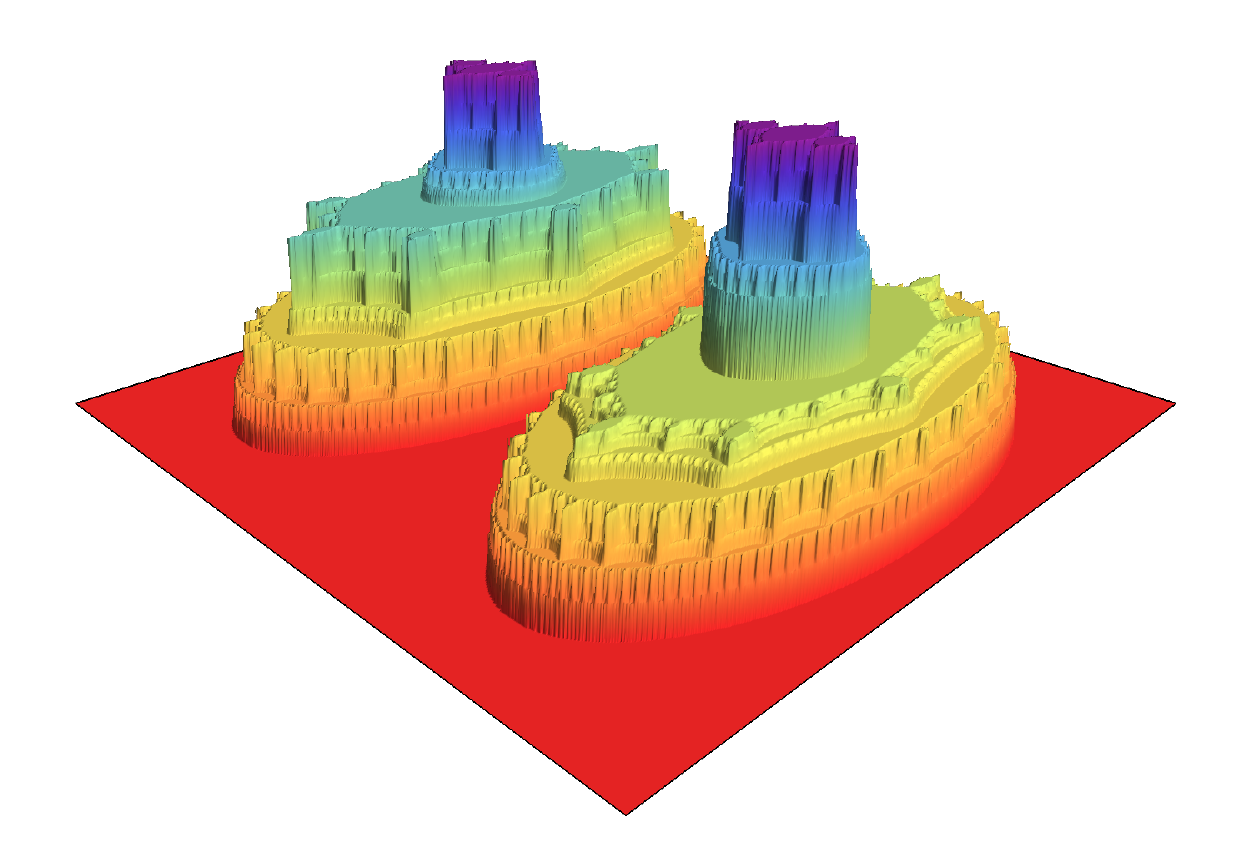}
\end{center}
 \caption{The graph of $1 - T_{\infty ,\tau }$ with a  new phenomenon which cannot hold in i.i.d. random dynamical systems of polynomials.}
\label{fig-IntromrkvT}
\end{minipage}
\hspace{5mm}
\begin{minipage}{72truemm}
  \begin{center}
 \includegraphics[width=4cm]{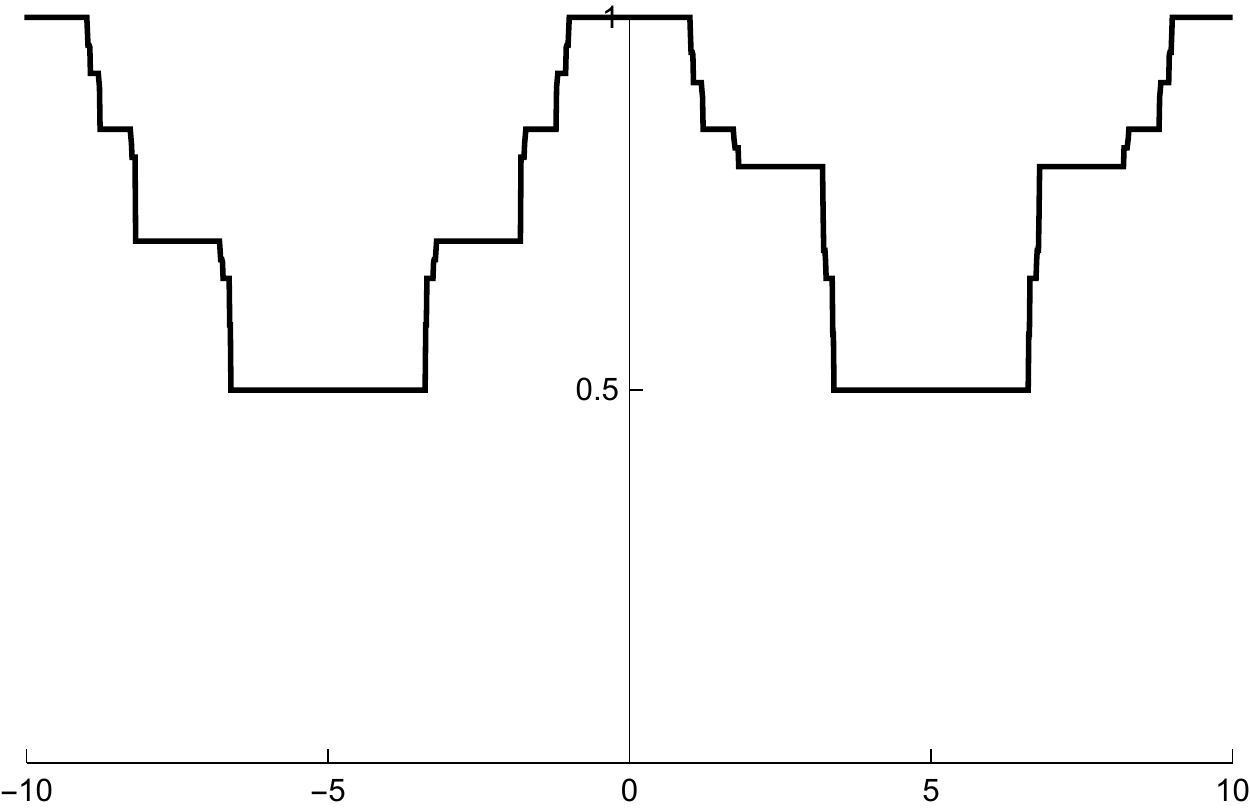}
 \end{center}
 \caption{The graph of the function $T_{\infty ,\tau }$ on the real line}
   \label{fig-section}
  \end{minipage}
\end{figure}

\subsection{Organization}
This paper is organized as follows.
In Section \ref{pre sec}, we introduce graph directed Markov systems on a compact metric space $(Y, d)$ and discuss some basic concepts.
Although we are most interested in holomorphic dynamics, we first treat such systems
on  general  compact metric spaces
 in order to show more generality.
The concept of graph directed Markov systems is similar to that of  rational semigroups in the theory of i.i.d.\ random holomorphic dynamical systems.
We define some kinds of Julia sets and show fundamental properties. 
We also discuss the dynamics of Markov operators following \cite{sumi11}.
In Section \ref{setting_sec}, we consider Markov random dynamical systems  that are induced by  given families $\tau$ of measures;
we define probability measures on the space of infinite product of $\cm (Y)$ and define a Markov operator $M_\tau$ induced by $\tau$.
Furthermore,  we prove Main Result \ref{res-stable}.
In Section \ref{ratMar_sec}, we focus on holomorphic dynamical systems on the Riemann sphere $\rs$.
We consider {\it rational} graph directed Markov systems in subsection \ref{julia_sec}, and prove Main Result \ref{res-dense} and other fundamental properties.
In subsection \ref{probfunc_section}, we investigate {\it polynomial} Markov random dynamical systems and prove Main Results \ref{res_Tconti}, \ref{res_addnoise}, \ref{res-T}   and \ref{res-6}.

\section*{Acknowledgment}
The authors thank Rich Stankewitz for valuable comments. The first author is partially sup- ported by JSPS Grant-in-Aid for Scientific Research (B) Grant Number JP 19H01790 and JSPS Grant-in-Aid for Scientific Research (A) Grant Number JP 18H03671. The second author is partially supported by JSPS Grant-in-Aid for JSPS Fellows Grant Number JP 19J11045.

\section{Preliminaries}\label{pre sec}
In this section, we introduce graph directed Markov systems on a compact metric space $(Y, d)$ and discuss some basic concepts.
These systems are similar to rational semigroups in the theory of i.i.d.\ random holomorphic dynamical systems.
In subsection \ref{julia_subsec}, we define some kinds of Julia sets and show fundamental properties of them. 
In subsection \ref{skewprod_subsec}, we give the definition of  skew product maps associated with graph directed  Markov systems and consider its dynamics.
In subsection \ref{mrkvop_subsec}, we discuss dynamics of Markov operators following \cite{sumi11}.

\subsection{Julia sets of graph directed Markov systems}\label{julia_subsec}
\begin{notation}\label{cm_notation}
We denote by $\cm (Y)$ the set of all continuous maps from $Y$ to itself and we define a metric $\kappa$ on $\cm (Y)$ by $\kappa (f,g):= \sup_{y \in Y} d(f(y),g(y))$.  
The space $\cm (Y)$ is a separable complete metric space since $Y$ is compact.
We denote by $\ocm (Y)$ the set of all open continuous maps from $Y$ to itself.
\end{notation}

\begin{definition}
Let $(V,E)$ be a directed graph with finite vertices and finite edges, and let $\Gamma_e$ be a non-empty subset of $\cm (Y)$ indexed by a directed edge $e \in E$.
We call $S = (V, E,(\Gamma_e)_{e \in E})$ a graph directed Markov system (GDMS for short) on $Y$.
The symbol $i(e)$ (resp. $t(e)$) denotes the initial (resp. terminal) vertex of each directed edge $e \in E$.   
\end{definition}

 \begin{figure}[h]
 \centering
 \includegraphics[width=5cm]{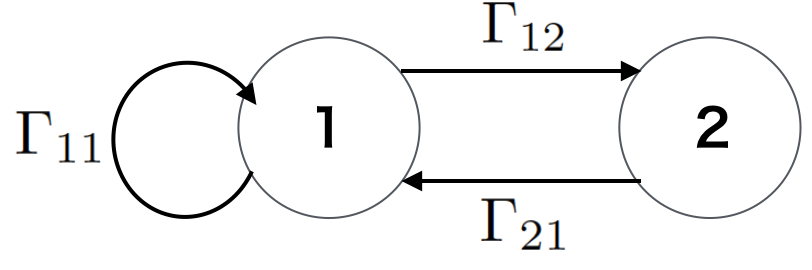}
\label{graph}
\caption{A schematic diagram of a GDMS}
\end{figure}

For an example of GDMS, see Figure 3. 
In the following, $S = (V, E,(\Gamma_e)_{e \in E})$ denotes a GDMS on $Y$.

\begin{definition}
Let $S= (V, E,(\Gamma_e)_{e \in E})$ be a GDMS.
\begin{enumerate}
\item A word  $e = (e_1,e_2, \dots, e_N) \in E^N$ with length $N \in \nn $ is said to be admissible if $t(e_n) = i(e_{n+1})$  for all  $n =1,2, \dots , N-1$. 
For this word $e$, we call $i(e_1)$ (resp. $t(e_N)$) the initial (resp. terminal) vertex of $e$ and we denote it by $i(e)$ (resp. $t(e)$). 
\item We set 
\begin{align*}
H(S)&:= \{f_{N} \circ \dots  \circ f_{2} \circ f_{1} ;\\
& N \in \nn, f_{n} \in \Gamma_{e_n}, t(e_n)=i(e_{n+1})(\forall n=1,\ldots, N-1)\}, \\
H_i(S)&:= \{f_{N} \circ  \dots \circ f_{2} \circ f_{1} \in H(S) ;\\
& N \in \nn, f_{n} \in \Gamma_{e_n}, t(e_n)=i(e_{n+1})(\forall n=1,\ldots, N-1),  i=i(e_1)\}, \\
H_i^j(S) &:= \{f_{N} \circ  \dots \circ f_{2} \circ f_{1} \in H(S) ;\\
& N \in \nn, f_{n} \in \Gamma_{e_n}, t(e_n)=i(e_{n+1})(\forall n=1,\ldots, N-1),  i=i(e_1), t(e_N)=j\}. 
\end{align*}
\end{enumerate}
\end{definition}

Now we define the Fatou sets and Julia sets of  GDMSs.
Recall that a subset $\mathcal{F} \subset \cm (Y)$ is said to be equicontinuous at a point $y \in Y$ if for every $\varepsilon >0$ there exists $\delta > 0$ such that for every $ f \in \mathcal{F}$ and for every $z \in Y$ with $d(y,z) < \delta$, we have $d(f(y),f(z))< \varepsilon$.
A subset $\mathcal{F} \subset \cm (Y)$ is said to be equicontinuous on a subset  $U \subset Y$ if  $\mathcal{F} $ is equicontinuous at every point of $U$.
See \cite{bear}.

\begin{definition}\label{julia_def}
Let $S= (V, E,(\Gamma_e)_{e \in E})$ be a GDMS.
\begin{enumerate}
\item  We denote by $F(S)$ the set of  all points $y \in Y$ for which there exists a neighborhood $U$ in $Y$ such that  the family  $H(S)$ is  equicontinuous on $U$.
 $F(S)$ is called the Fatou set of $S$ and the complement $J(S) := Y \setminus F(S)$ is called the Julia set of $S$.
\item For each $i \in V$, we denote by $F_i(S)$  the set of  all points $y \in Y$ for which there exists a neighborhood $U$ in $Y$ such that  the family  $H_i(S)$ is  equicontinuous on $U$.
 $F_i(S)$ is called the Fatou set of $S$ at the vertex $i$  and the complement $J_i(S) := Y \setminus F_i(S)$ is called the Julia set of $S$ at the vertex $i$.
\item Set $\mathbb{F}(S) :=\bigcup_{i \in V} F_i(S) \times \{i\},\mathbb{J}(S) :=\bigcup_{i \in V} J_i(S) \times \{i\}$.
\end{enumerate}
\end{definition}

\begin{rem}\label{semiJulia_rem}
Let $S= (V, E,(\Gamma_e)_{e \in E})$  be a GDMS with just one vertex and just one edge, say $V=\{ 1 \}, E= \{ (1,1)\}$. Then $H(S)=H_1(S)$ coincides with the semigroup generated by $\Gamma_{(1,1)}$, where the product is the  composition of maps,  and the Fatou (resp.\ Julia) set of the GDMS $S$ coincides with the Fatou (resp.\ Julia) set of the semigroup $H(S)$ of continuous maps on $Y$.
By definition, the Fatou set of a semigroup $G$ of continuous maps on $Y$ is the set of  all points $y \in Y$ for which there exists a neighborhood $U$ in $Y$ such that  the family  $G$ is  equicontinuous on $U$.
The Fatou sets of  semigroups are related to i.i.d. random dynamical systems.
See \cite{stank12}, \cite{sumi11}. 
\end{rem}

\begin{rem}
The Fatou sets $F(S)$ and $F_i(S)$ are open subsets of $Y$ and  the Julia sets $J(S)$ and $J_i(S)$ are compact subsets of $Y$.
Moreover, we have $F(S)= \bigcap_{i \in V} F_i(S)$ and $J(S)=\bigcup_{i \in V} J_i(S)$.
\end{rem}

\begin{notation}
For families $\mathcal{F}_1, \mathcal{F}_2 \subset \cm (Y)$  of maps, define \[ \mathcal{F}_2 \circ \mathcal{F}_1 := \{ f_2 \circ f_1\in \cm (Y)\, ;\, f_1 \in \mathcal{F}_1  \text{ and } f_2 \in \mathcal{F}_2 \}.\] 
For an element $h\in \mathrm{CM}(Y)$ and ${\mathcal F}\subset \mathrm{CM}(Y)$, we set $h\circ {\mathcal F}:= \{ h\} \circ {\mathcal F}.$
\end{notation}

\begin{notation}
$B_d (y,r)$ denotes the ball of  radius $r$ and center  $y$ with  respect to the metric $d$ in the space $Y$.
 It is also denoted  briefly by $B (y,r)$.
 \end{notation}

\begin{lemma}\label{equiconti_lemma}
Suppose that $Y$ is locally connected, 
that $\mathcal{F} \subset \cm (Y)$  is not equicontinuous at a point  $y_0 \in Y$ and that 
$h \in  \cm (Y)$ satisfies the following condition.
\begin{enumerate}
\item[]  $\sup \{ \diam C ; \,C \text{ is a connected component of } h^{-1}(B (y,\varepsilon))\} \to 0$ as $\varepsilon \to 0$ for any point $y \in Y$. 
\end{enumerate}
Then $h \circ \mathcal{F} $ is  not equicontinuous at the point  $y_0$.
\end{lemma}

\begin{proof}
We first note that the following.
Since $\mathcal{F}$  is not equicontinuous at a point  $y_0 \in Y$, there exists a positive real number $\epsilon_0$ such that for any $\delta_0 >0$ there exists $y \in B(y_0, \delta_0)$ and $f \in \mathcal{F}$ such that $d(f(y),f(y_0)) \geq \epsilon_0$.
By the assumption, for each $z \in Y$, there exists $\varepsilon >0$ such that the diameter of each connected component of $h^{-1}\left(B (z,\varepsilon) \right)$ is less than $\epsilon_0$. 
Since $Y$ is compact, we may assume that $\varepsilon$ does not depend on $z$.

The proof is by contradiction. 
Suppose that  $h \circ \mathcal{F} $ is equicontinuous at $y_0$. 
Then there exists $\delta >0$ such that for any $y \in B(y_0, \delta)$ and for any $f \in \mathcal{F}$ such that $d(h \circ f(y),h \circ f (y_0)) < \varepsilon$.
Since $Y$ is locally connected, we may assume that $B(y_0, \delta)$ is connected.
By the definition of $\epsilon_0$, there exists $y_1 \in B(y_0, \delta)$ and $g \in \mathcal{F}$ such that $d (g(y_1),  g(y_0)) \geq \epsilon_0$.
Let $C$ be the connected component of $h^{-1}\left(B (h \circ g(y_0),\varepsilon) \right)$ which contains $g(y_0)$, 
whose diameter  is necessarily less than $\epsilon_0$.
Since $B(y_0, \delta)$ is connected and $h \circ g ( B(y_0, \delta)) \subset B (h \circ g(y_0),\varepsilon) $, we have $g ( B(y_0, \delta)) \subset C$, and hence $g(y_1) \in C$.
However, this contradicts the fact that $d (g(y_1),  g(y_0)) \geq \epsilon_0$ and  $\diam C < \epsilon_0$.
\end{proof}

\begin{definition}
A GDMS $S= (V, E,(\Gamma_e)_{e \in E})$  is said to be {\it irreducible} if the directed graph of $S$ is strongly connected, that is, for any $(i,j) \in V \times V$, there exists an admissible word $e$ such that $i=i(e)$ and $t(e)=j$.
\end{definition}

\begin{lemma}\label{sysJulia=semiJulia_lemma}
Let $S= (V, E,(\Gamma_e)_{e \in E})$  be an irreducible GDMS such that every $h \in H(S)$ satisfies the condition mentioned in Lemma \ref{equiconti_lemma}: $$\sup \{ \diam C ; \,C \text{ is a connected component of } h^{-1}(B (y,\varepsilon))\} \to 0$$ as $\varepsilon \to 0$ for any point $y \in Y$. 
Then $J_i(S) = J(H_i^i(S) ) $  for any $i \in V$, where $J(H_i^i(S) ) $ is the Julia set of the semigroup $H_i^i(S)$.
\end{lemma}

\begin{proof}
Since $J_i(S)  \supset J(H_i^i(S) ) $  trivially, we show $J_i(S)  \subset J(H_i^i(S) ) $.
For any  $y_0 \in J_i(S)$, there exists a vertex $j \in V$ such that $H_i^j(S)$ is not equicontinuous on any neighborhood of  $y_0$ since $\# V < \infty$.
We fix  $h \in H_j^i(S)$, which exists  by the irreducibility of $S$.
According to Lemma \ref{equiconti_lemma}, $h \circ H_i^j(S)$ is not equicontinuous  on any neighborhood of   $y_0$.
Thus $H_i^i(S)$ is not equicontinuous  on any neighborhood of   $y_0$, 
and hence $y_0 \in  J(H_i^i(S) ) $.
\end{proof}

\begin{rem}\label{sysJulia=semiJulia_rem}
If $Y = \rs$ and  $\Gamma_e \subset \rat$ for all $e \in E$, then every $g \in H(S)$ satisfies the condition mentioned in Lemma \ref{equiconti_lemma}.
Thus Lemma \ref{sysJulia=semiJulia_lemma} holds in this case.
\end{rem}

\begin{notation}
For a family $\mathcal{F} \subset \cm (Y)$ and a set $X \subset Y$, we set  $$\mathcal{F} (X) := \bigcup_{f \in \mathcal{F}} f(X),\,
\mathcal{F} ^{-1}(X) :=\bigcup_{f \in\mathcal{F}} f^{-1}(X) .$$
If $ \mathcal{F} = \emptyset$, then we set $\mathcal{F} (X) := \emptyset, \,
\mathcal{F} ^{-1}(X) := \emptyset .$
\end{notation}

\begin{definition}
Let $L_i$ be a subset of $Y$ for each $i \in V$. 
We consider the familiy  $(L_i)_{i \in V}$  indexed by $i \in V$.
\begin{enumerate}
\item $(L_i)_{i \in V}$ is said to be {\it forward $S$-invariant} if $\Gamma_e( L_{i(e)}) \subset L_{t(e)}$  for all $e \in E$.
\item $(L_i)_{i \in  V}$ is said to be {\it backward $S$-invariant} if $\Gamma_e^{-1}( L_{t(e)}) \subset L_{i(e)}$   for all $e \in E$. 
\end{enumerate}
\end{definition}

It is easy to prove the following lemma and the proof is left to the
readers. 

\begin{lemma}\label{inv_lemma}
\begin{enumerate}
\item If a GDMS $S= (V, E,(\Gamma_e)_{e \in E})$  satisfies $\Gamma_e  \subset \ocm(Y)$ for all $e \in E$, then the family $(F_i (S))_{i \in V}$ (resp. $(J_i(S))_{i \in V}$) of Fatou sets (resp. Julia sets) is forward (resp. backward) $S$-invariant.
\item If  $(L_i)_{i \in V}$ is forward $S$-invariant, then $ H_i^j (S)( L_{i}) \subset L_{j}$  for every $i , j \in V$.
If  $(L_i)_{i \in V}$ is backward $S$-invariant, then $ (H_i^j(S))^{-1} ( L_{j}) \subset L_{i}$  for every $i , j \in V$.
\item  Let  $S$ be an irreducible GDMS and let $(L_i)_{i \in V}$ be a forward $S$-invariant family.
Then $L_i =\emptyset$ for {\it all} $i \in V$ if and only if $L_j =\emptyset$ for {\it some} $j \in V$.
\end{enumerate}
\end{lemma}

\begin{prop}\label{selfsim_prop}
If $\Gamma_e$ is a compact subset of $\ocm (Y)$  for all $e \in E$, then  $$\bigcup_{e \in E \colon i(e) =i } \Gamma_e^{-1}( J_{t(e)}(S)) = J_{i}(S)$$  for all $i \in V$.
\end{prop}

\begin{proof}
If there is no $e \in E$ that satisfies $i(e) = i$, the statement is trivial.
Hence we may assume  there exists some $e \in E$ that satisfies $i(e) = i$.

According to Lemma \ref{inv_lemma}, $\bigcup_{i(e) =i } \Gamma_e^{-1}( J_{t(e)}(S)) \subset J_{i}(S) $.
Fix any $y \notin \bigcup_{i(e) =i } \Gamma_e^{-1}( J_{t(e)}(S))$.
Since $E$ is finite and $\Gamma_e$ is compact for all $e \in E$, we have $y \notin J_{i}(S)$.
Thus $\bigcup_{i(e) =i } \Gamma_e^{-1}( J_{t(e)}(S)) \supset J_{i}(S) $.
\end{proof}

\begin{definition}
We set $J_{\ker,i} (S):=  \bigcap_{ j \in V :  H_i^j(S) \neq \emptyset} \bigcap_{h \in H_i^j(S) } h^{-1} (J_j(S))$
and  call it the kernel Julia set of $S$ at the vertex $i \in V$.
Here, we set $J_{\ker,i} (S):= \emptyset$ if $H_i^j(S)= \emptyset$ for all $j \in V$.
Recall that the kernel Julia set of a semigroup $G \subset \cm (Y)$ is defined by $J_{\ker} (G) = \bigcap_{g \in G} g^{-1}(J(G))$, where $J(G)$ denotes the Julia set of $G$ defined in Remark \ref{semiJulia_rem} (see \cite{sumi11}).
Moreover, we set  $\mathbb{J} _{{\rm ker}} (S) := \bigcup_{i \in V} J_{{\rm {\rm ker}},i}(S) \times \{ i \} \subset Y \times V$.
\end{definition}

\begin{cor}
Suppose that an irreducible  GDMS $S= (V, E,(\Gamma_e)_{e \in E})$  satisfies   $\Gamma_e \subset \ocm(Y)$ for all $e \in E$ and 
 every $h \in H(S)$ satisfies the condition mentioned in Lemma \ref{equiconti_lemma}: $$\sup \{ \diam C ; \,C \text{ is a connected component of } h^{-1}(B (y,\varepsilon))\} \to 0$$ as $\varepsilon \to 0$ for any point $y \in Y$. 
Then $J_{{\rm ker} , i} (S) = J_{{\rm ker}} (H_i^i (S))$, 
 where the right hand side is the kernel Julia set of the semigroup $H_i^i(S)$.
\end{cor}

\begin{proof}
By Lemma \ref{sysJulia=semiJulia_lemma}, we have $J_{{\rm ker} , i} (S) \subset  J_{{\rm ker}} (H_i^i(S))$.
We now fix   $ y \in J_{{\rm ker}} (H_i^i(S))$ and fix  $h \in H_i^j(S)$.
Since $S$ is irreducible, there exists   $f \in H_j^i(S)$ so that  $f \circ h \in H_i^i(S)$ and hence $f \circ h (y) \in J(H_i^i (S))$.
Thus, we have $h (y) \in J_j(S)$ by Lemma \ref{sysJulia=semiJulia_lemma} and Lemma \ref{inv_lemma}.
\end{proof}

\begin{notation}
Let $(L_i)_{i \in  V}, (\tilde{L}_i)_{i \in  V}$ be families of subsets of $Y$ indexed by $V$.
We write  $(L_i)_{i \in  V} \subset  (\tilde{L}_i)_{i \in  V}$  if $L_i \subset \tilde{L}_i$ for all $i \in V$.
\end{notation}

\begin{lemma}\label{kerjulia_lemma}
For kernel Julia sets $(J_{\ker , i}(S) )_{i \in V}$, the following statements hold.
\begin{enumerate}
\item The family  $(J_{\ker , i}(S) )_{i \in V}$ is forward $S$-invariant.
\item If a forward $S$-invariant  family $(L_i)_{i \in  V}$ satisfies $(L_i)_{i \in  V} \subset (J_i(S))_{i \in V}$, then  $(L_i)_{i \in  V} \subset (J_{\ker, i}(S))_{i \in V}$.
\end{enumerate}
\end{lemma}

The proof is immediate by using Lemma \ref{inv_lemma}.
Now we define a condition that  plays an important role in section \ref{probfunc_section}.

\begin{definition}\label{sepcond_def}
We say that a GDMS $S= (V, E,(\Gamma_e)_{e \in E})$  satisfies the {\it backward separating condition} if $f_{1} ^{-1} (J_{t(e_1)}(S) ) \cap f_{2} ^{-1} (J_{t(e_2)}(S) ) = \emptyset$ for every  $e_1, e_2 \in E$ with the same initial vertex and for every $f_{1}\in \Gamma_{e_1}, f_{2} \in \Gamma_{e_2} $, except the case $e_1 =e_2$ and $f_{1} =f_{2}  $.
\end{definition}

\begin{definition}
	We say that a GDMS $S= (V, E,(\Gamma_e)_{e \in E})$ is {\it essentially non-deterministic} if there exist $e_1, e_2 \in E$ with $i(e_1)=i(e_2)$ and exist $f_1 \in \supp \tau_{e_1}, f_2 \in \supp \tau_{e_2}$ such that either $e_1 \neq e_2$ or $f_1 \neq f_2$.
\end{definition}

\begin{lemma}\label{sepcond_lemma}
Let $S= (V, E,(\Gamma_e)_{e \in E})$ be a GDMS which satisfies the backward separating condition.
If $S$ is essentially non-deterministic,
then   $J_{{\rm ker},j}(S) = \emptyset$ for some $j \in V$. 
Moreover, if, in addition to the assumption above, $S$ is  irreducible, then $\Bbb{J}_{{\rm ker} }(S)=\emptyset $.
\end{lemma}

\begin{proof}
Since $S$ is essentially non-deterministic, there exist $e_1, e_2 \in E$ with $i(e_1)=i(e_2)=: j$ and exist $f_1 \in \supp \tau_{e_1}, f_2 \in \supp \tau_{e_2}$ such that either $e_1 \neq e_2$ or $f_1 \neq f_2$.
If there exists some $z \in J_{\ker, j}(S)$, then $f_n(z) \in J_{t(e_n)}(S)\ (n =1,2)$ by definition.
However, this implies that $f_{1} ^{-1} (J_{t(e_1)} (S))$ and $f_{2} ^{-1} (J_{t(e_2)} (S))$ share a point $z$, which contradicts  the backward separating condition.
If $S$ is irreducible, then $\Bbb{J}_{{\rm ker} }(S)=\emptyset $ by Lemma \ref{inv_lemma} and Lemma \ref{kerjulia_lemma}.
\end{proof}
\subsection{Skew product maps}\label{skewprod_subsec}
Let  $S= (V, E,(\Gamma_e)_{e \in E})$ be a GDMS on $Y$.
We define the skew product map associated with $S$ and investigate its dynamics.
We consider only admissible maps as in subsection \ref{julia_subsec}.

\begin{definition}\label{xi_nota}
We say that  a sequence $\xi = (\gamma_n , e_n )_{n =1}^N  \in (\cm(Y) \times E )^{N}$ is admissible with length $N$ if $e=(e_1, \dots , e_N) $ is admissible and $\gamma_n \in \Gamma_{e_n}$ for all $n \in \{1, \dots, N\}$.
Also, we say that a sequence $\xi = (\gamma_n , e_n )_{n \in \nn}  \in (\cm(Y) \times E )^{\nn}$  is admissible  if  $\gamma_n \in \Gamma_{e_n}$ and  $t(e_n)=i(e_{n+1})$ for all $n \in \nn$.
For any admissible sequence $\xi = (\gamma_n , e_n )_{n \in \nn}$ and for any $N, M \in \nn$ with $N>M$, we set $\gamma_{N,M}:=\gamma_N \circ \cdots \circ \gamma_M$ and  $\xi_{N,M} := (\gamma_n , e_n )_{n =M}^N $.
Let  $h=(\gamma_n,e_n)_{n=1}^N$ be an admissible sequence.
We regard $h$ as a map from $Y \times \{ i(e_1)\}$ to $Y\times \{ t(e_N)\}  $ by setting $h(y,i(e_1) ) :=(\gamma_{N,1}(y) , t(e_N)) \quad( y \in Y)$.
\end{definition}

\begin{definition}
We define the set of all admissible infinite sequences by
\begin{align*}
X(S):=\{ \xi = (\gamma_n , e_n )_{n \in \nn} & \in (\cm(Y) \times E )^{\nn} ;
 &\gamma_n \in \Gamma_{e_n} \text{ and } t(e_n)=i(e_{n+1}) \text{ for all } n \in \nn \}. 
\end{align*}
We denote by $X_i(S)$ the subset of $X(S)$ consisting of all admissible infinite sequences with initial vertex $i \in V$; thus any $\xi \in X_i(S)$ can be written as $\xi = (\gamma_n , e_n )_{n \in \nn}$ such that  $i(e_1)=i, \gamma_n \in \Gamma_{e_n} \text{ and } t(e_n)=i(e_{n+1}) \text{ for all } n \in \nn$.
We endow  $ (\cm(Y) \times E )^{\nn}$ with  the product topology, where $E$ has the discrete topology.
We endow $X(S)$ and $X_i(S)$ with the relative topology from  $ (\cm(Y) \times E )^{\nn}$.
Note that $X(S)$ and $X_i(S)$ are compact if $\Gamma_e$ is compact for all $e \in E$.
\end{definition}

\begin{definition}
For any $\xi = (\gamma_n , e_n )_{n \in \nn}  \in X(S)$, we denote by $F_\xi$ the set  of all points $y \in Y$ for which there exists a neighborhood $U$ in $Y$ such that the family of maps $\{\gamma_{N,1} = \gamma_N \circ \cdots \circ \gamma_1 ; \,N \in \nn\}$ is equicontinuous on $U$.
We call $F_\xi$ the Fatou set of $\xi$ and call the complement $J_\xi := Y \setminus F_\xi$  the Julia set of $\xi$.
Set $F^\xi : = \{\xi \} \times F_\xi  \subset X(S) \times Y $ and  $J^\xi := \{\xi \} \times J_\xi  \subset X(S) \times Y$.
\end{definition}

\begin{lemma}\label{limsupdiam_lemma}
\begin{enumerate}
\item For any $ \xi  \in X_i(S)$, we have  $J_\xi \subset J_i(S)$. 
\item We have $$\bigcup_{\xi \in X(S) } F^\xi \subset \{ (\xi , y ) \in   X(S) \times Y ; \lim_{\varepsilon \to 0} \sup_{n \in \nn} {\rm diam} (\gamma_{n,1} B(y , \varepsilon) )=0 \} .$$
\item For any $ \xi = (\gamma_n , e_n )_{n \in \nn} \in X_i (S)$, we have $J_\xi \subset \bigcap_{n \in \nn } \gamma_{n , 1} ^{-1} (J_{t(e_n)} (S) ) $.
\end{enumerate}
\end{lemma}

\begin{proof}
\begin{enumerate}
\item For any $\xi= (\gamma_n , e_n )_{n \in \nn} \in X_i(S)$, we have $\{ \gamma_N \circ \cdots \circ \gamma_1 ; \, N \in \nn\} \subset H_i(S).$
Thus $J_\xi \subset J_i(S)$. 
\item For any  $(\xi, y) \in F^\xi$, we have $y \in F_\xi$ and
set $ \xi = (\gamma_n , e_n )_{n \in \nn}$.
For any $\eta >0 $, there exists $\delta > 0$ such that $d(\gamma_{n,1}(y) , \gamma_{n,1}(y') ) < \eta$  for any $y' \in B(y,\delta)$ and any $n \in \nn$.
Now we take  $\varepsilon < \delta$.
Then \[ d(\gamma_{n,1}(y_1) , \gamma_{n,1}(y_2) ) \leq d(\gamma_{n,1}(y_1) , \gamma_{n,1}(y) ) + d(\gamma_{n,1}(y) , \gamma_{n,1}(y_2) ) < 2 \eta \]
 for any $y_1 , y_2 \in B(y , \varepsilon)$.
\item Suppose $\gamma_{m,1} (y) \in F_{t(e_m)}(S)$ for some $y \in J_\xi$ and $m \in \nn$.
Then  $\gamma_{m,1}^{-1}(F_{t(e_m)}(S))$ is a neighborhood of $y$.
Now $\{ \gamma_{n,m+1} \} _{n >m} \subset H_{t(e_m)}(S) $ implies that  $\{ \gamma_{n,1} \} _{n \in \nn} $ is equicontinuous on  $\gamma_{m,1}^{-1}(F_{t(e_m)}(S))$.
This contradicts the hypothesis that $y \in J_\xi$.
\end{enumerate}
\end{proof}

\begin{rem}
If $Y = \rs$ and  $\Gamma_e \subset \rat$ for all $e \in E$, then the equality holds in the statement of Lemma \ref{limsupdiam_lemma}(ii), where $\rat$ denotes the space of all non-constant rational maps.
\end{rem}

\begin{definition}
Let $S= (V, E,(\Gamma_e)_{e \in E})$ be a GDMS and
let $\sigma \colon X(S) \to X(S)$ be the (left) shift map.
Define the skew product map  $\tilde{f} : X(S) \times Y \to X(S) \times Y$ associated with $S$, by $\tilde{f}( \xi , y) = (\sigma (\xi) , \gamma_1 (y) )$ for any $\xi = (\gamma_n , e_n )_{n \in \nn} \in X(S)$ and any $y \in Y$.
Also we set $\tilde{J}(\tilde{f}) := \overline{\bigcup_{\xi \in X(S) }J^\xi}$, where the closure is taken in the product space $X(S) \times Y$, and we call this the skew product Julia set of $\tilde{f}$.
\end{definition}

\begin{lemma}\label{skewprod0_lemma}
The skew product map $\tilde{f}$ is continuous on $X(S) \times Y$ and $ J_\xi \subset \gamma_1^{-1} (J_{\sigma(\xi)})$ holds for any  $\xi = (\gamma_n , e_n )_{n \in \nn} \in X(S)$.
If  $\gamma_1$ is an open map, then $ J_\xi = \gamma_1^{-1} (J_{\sigma(\xi)})$.
\end{lemma}

\begin{lemma}
With above terminology, we have $\tilde{J} (\tilde{f}) \subset \tilde{f}^{-1}(\tilde{J}(\tilde{f}))$.
If  $\Gamma_e \subset \ocm(Y)$ for all  $ e \in E$, then $\tilde{f}$ is open and $\tilde{J} (\tilde{f})  = \tilde{f}^{-1}(\tilde{J}(\tilde{f}))$  holds, where $\ocm(Y)$ denotes the set of all open continuous maps $Y \to Y$.
\end{lemma}

\begin{proof}
By Lemma \ref{skewprod0_lemma}, we have $\tilde{f} (\tilde{J} (\tilde{f}) )  \subset \overline{\bigcup_{\xi \in X (S)} \tilde{f} (J^\xi )} \subset  \overline{\bigcup_{\xi \in X(S) } J^{ \sigma( \xi)} } \subset \tilde{J}(\tilde{f})$.
Thus, $\tilde{J}(\tilde{f}) \subset  \tilde{f}^{-1}(\tilde{J} (\tilde{f}) ) $.
If  $\Gamma_e \subset \ocm(Y)$ for all  $ e \in E$, then $\tilde{f}$ is open.
Combining this with   Lemma \ref{skewprod0_lemma},  we have  $\tilde{J}(\tilde{f}) \supset  \tilde{f}^{-1}(\tilde{J} (\tilde{f}) ) $.
\end{proof}
\subsection{Markov operators}\label{mrkvop_subsec}
Throughout this subsection, let $\mathbb{Y}$ be a compact metric space.
We discuss Markov operators on the space ${\rm C}(\mathbb{Y})$ of all  continuous complex functions on  $\mathbb{Y}$.
The space ${\rm C}(\mathbb{Y})$ is a Banach space with the supremum norm $\| \cdot \|_\mathbb{Y}$ and its normed dual ${\rm C}(\mathbb{Y})^*$  can be regarded as the set of all regular complex Borel measures on $\mathbb{Y}$ by the theorem of F. and M. Riesz.

\begin{notation}
We denote by $\mathfrak{M}_1(\mathbb{Y})$ the set of all regular Borel probability measures on $\mathbb{Y}$ and we define the weak*-topology on $\mathfrak{M}_1(\mathbb{Y}) \subset {\rm C}(\mathbb{Y})^*$.
Namely, $\mu_n \to \mu $ if and only if $\mu_n(\phi) \to \mu(\phi)$ for all $\phi \in {\rm C}(\mathbb{Y})$,
where we write $\mu(\phi) := \int_{\mathbb{Y}} \phi \, {\rm d} \mu$ for any $\mu \in \mathfrak{M}_1(\mathbb{Y} )$ and any $\phi \in {\rm C}(\mathbb{Y})$.
The space $\mathfrak{M}_1(\mathbb{Y})$ is compact by the Banach-Alaoglu theorem.
\end{notation}

\begin{rem}\label{metronm1y_rem}
We introduce a metric $d_0$ on $\mathfrak{M}_1(\mathbb{Y})$  as follows.
Take a countable dense subset  $\{ \phi_k \}_{k \in \nn}$ of ${\rm C}(\mathbb{Y})$ whose existence is due to the compactness of $\mathbb{Y}$.
We define the distance between two points $\mu_1, \mu_2 \in \mathfrak{M}_1(\mathbb{Y} )$ by $$d_0 (\mu_1,\mu_2) := \sum_{k \in \nn} \frac{1}{2^k} \frac{|\mu_1(\phi_k) - \mu_2(\phi_k)|}{1 + |\mu_1(\phi_k) - \mu_2(\phi_k)|}.$$
\end{rem}

\begin{definition}\label{mrkvop_def}
A linear operator $M \colon  {\rm C}( \mathbb{Y}) \to  {\rm C}( \mathbb{Y})$ is called a Markov operator if $M \1_{\mathbb{Y}} =\1_{\mathbb{Y}}$ and  $M \phi \geq 0$ for all $\phi \in  {\rm C}( \mathbb{Y})$ with $ \phi \geq 0$,
where we write $\psi \geq 0$ if $\psi(y)$ is non-negative real number  for all $y \in \mathbb{Y}$.
\end{definition}

\begin{lemma}
The operator norm of a Markov operator  $M \colon  {\rm C}( \mathbb{Y}) \to  {\rm C}( \mathbb{Y})$ is equal to one.
Thus the adjoint $M^* \colon  {\rm C}( \mathbb{Y})^* \to  {\rm C}( \mathbb{Y})^*$ satisfies that 
 $M^*(\mathfrak{M}_1(\mathbb{Y})) \subset \mathfrak{M}_1(\mathbb{Y})$, where
$$(M^* \mu ) \phi := \mu (M \phi  ), \quad \mu \in  {\rm C}( \mathbb{Y})^*, \phi \in  {\rm C}( \mathbb{Y}) .$$
\end{lemma}

\begin{proof}
Since $M \1_{\mathbb{Y}} =\1_{\mathbb{Y}}$,  the operator norm $\| M\| \geq 1$.
For any  $\phi \in  {\rm C}( \mathbb{Y})$ with $\| \phi \|_{\mathbb{Y}} \leq 1$,
we have $ 0 \leq |\phi | \leq \1$.
Fix any $y \in \mathbb{Y}$ and define $A:=M (| \phi | ^2)(y), B:= |M\phi(y)|$.
By the above  properties  of $M$, we have $A \leq 1$.
On the other hand, there exists a complex number $\alpha$ with modulus $1$ such that $\alpha M\overline{\phi}(y) =B$.
Then for any $t \in \mathbb{R}$, we have
\[ 0 \leq M( | \phi - t  \alpha |^2 )(y) \leq 1 -2 B\, t + t^2.\]
It follows that $ |M\phi(y)| = B \leq 1$ and $\|M \phi \|_{\mathbb{Y}} \leq 1$.
Hence $\| M \| =1$.
\end{proof}

\begin{rem}
For each $y \in \mathbb{Y}$, let  $\Phi (y)$  be the Dirac measure centered at $y$.
Note that $\Phi \colon \mathbb{Y} \to \mathfrak{M}_1 (\mathbb{Y} )$ is a topological embedding.
We regard $\mathbb{Y}$ as a subset of $ \mathfrak{M}_1 (\mathbb{Y} )$ by using  $\Phi$.
\end{rem}

\begin{definition}\label{fatouM_def}
For a Markov operator $M :  {\rm C}( \mathbb{Y}) \to  {\rm C}( \mathbb{Y})$, we consider the family $\{ (M^*)^n : \mathfrak{M}_{1}({\mathbb{Y}} )\to \mathfrak{M}_{1}({\mathbb{Y}}) \}_{n \in \nn} $  of iterations of the adjoint map $M^*$.
\begin{enumerate}
\item We denote by $F_{{\rm meas}} (M^*)$  the set of  all points $\mu \in  \mathfrak{M}_1 (\mathbb{Y} )$ for which there exists a neighborhood $\mathcal{U}$ in $\mathfrak{M}_1 (\mathbb{Y} )$ such that  the family $\{ (M^*)^n : \mathfrak{M}_{1}({\mathbb{Y}} )\to \mathfrak{M}_{1}({\mathbb{Y}}) \}_{n \in \nn} $   is  equicontinuous on $\mathcal{U}$.
\item We denote by $F_{{\rm meas}} ^0(M^*)$  the set of  all points $\mu \in  \mathfrak{M}_1 (\mathbb{Y} )$ satisfying that  the family $\{ (M^*)^n : \mathfrak{M}_{1}({\mathbb{Y}} )\to \mathfrak{M}_{1}({\mathbb{Y}}) \}_{n \in \nn} $   is  equicontinuous at $\mu$.
%
\item We denote by $F_{{\rm pt}} (M^*)$ the set of  all points $y \in  \mathbb{Y} $  for which there exists a neighborhood ${U}$ in $\mathbb{Y}$ such that  the family $\{ (M^*)^n |_{\mathbb{Y}}: \mathbb{Y} \to \mathfrak{M}_{1}({\mathbb{Y}}) \}_{n \in \nn} $ restricted to $\mathbb{Y} \subset \mathfrak{M}_1 (\mathbb{Y} )$  is equicontinuous on $U$.
\item We denote by $F_{{\rm pt}} ^0(M^*)$ the set of  all points $y \in  \mathbb{Y} $ satisfying that  the family $\{ (M^*)^n |_{\mathbb{Y}}: \mathbb{Y} \to \mathfrak{M}_{1}({\mathbb{Y}}) \}_{n \in \nn} $ restricted to $\mathbb{Y} \subset \mathfrak{M}_1 (\mathbb{Y} )$  is equicontinuous at $y$.
\end{enumerate}
\end{definition}

\begin{lemma}\label{fpt0_lemma}
For a Markov operator $M \colon {\rm C}( \mathbb{Y}) \to  {\rm C}( \mathbb{Y})$,
we have that $ y_0 \in  F_{{\rm pt}} ^0(M^*)$ if and only if $\{ M^n \phi \colon \mathbb{Y} \to \mathbb{C}\}_{n \in \nn} $ is equicontinuous at $y_0$ for all $\phi \in {\rm C}( \mathbb{Y})$.
\end{lemma}

\begin{proof}
Fix a countable dense  set $\{ \phi_k \}_{k \in \nn} \subset \mathrm{C} (\mathbb{Y})$ and let  $d_0$ be the metric  mentioned in Remark \ref{metronm1y_rem}.
Suppose that  $\{ M^n \phi \}_{n \in \nn} $ is equicontinuous at $y_0 \in \mathbb{Y}$ for all $\phi \in {\rm C}( \mathbb{Y})$. 
Take small $ \varepsilon > 0$.
Then there exists some $N \in \nn$ such that $\sum_{k=N+1}^{\infty} {2^{-k}} < \varepsilon$.
By our assumption, there exists $\delta >0$ such that for any $y \in B(y_0 , \delta) $, any $n \in \nn$ and any $k = 1, \dots , N$, we have $$| M^n \phi_k (y) -M^n \phi_k (y_0)| < \frac{\varepsilon / N}{1-\varepsilon / N}. $$
It follows that 
\begin{align*}
d_0( (M^*)^n ( \delta_y ),  (M^*)^n (\delta_{y_0})) &= \sum_{k \in \nn }  \frac{1}{2^k} \frac{| M^n\phi_k(y) - M^n \phi_k (y_0)|}{1 + | M^n \phi_k (y) - M^n \phi_k (y_0)|}   \leq   \varepsilon + \sum_{k=1}^N \varepsilon/N = 2\varepsilon ,
\end{align*}
and hence $y_0 \in F_{{\rm pt}} ^0(M^*)$.

Conversely, suppose that $y_0 \in F_{{\rm pt}} ^0(M^*)$, and take any $\phi \in {\rm C}(\mathbb{Y})$ and $\varepsilon > 0$.
Since $\{ \phi_k \}_{k \in \nn}$ is dense in ${\rm C}(\mathbb{Y})$,
there exists  $k \in \nn $ such that $\| \phi_k - \phi \|_{\mathbb{Y}} < \varepsilon$.
By $y_0 \in F_{{\rm pt}} ^0(M^*)$, there exists  $\delta > 0$ such that for any $y \in B(y_0 , \delta)$ and any  $n \in \nn$,
we have  $$d_0( (M^*)^n (\delta_y)  , (M^*)^n (\delta_{y_0}) ) < \frac{1}{2^k} \frac{\varepsilon}{ 1 + \varepsilon}. $$
It follows that $| M^n \phi_k (y) - M^n \phi_k (y_0) | < \varepsilon $ and
 \begin{align*}
 &\quad | M^n \phi(y) - M^n \phi (y_0) |  \\
  &\leq | M^n \phi(y) - M^n \phi_k(y) |+| M^n \phi_k(y) - M^n \phi_k (y_0) |+| M^n \phi_k(y_0) - M^n \phi(y_0) | 
 < 3 \varepsilon.
 \end{align*}
Thus $\{ M^n \phi \}_{n \in \nn} $ is equicontinuous at $y_0$.
\end{proof}

\begin{lemma}\label{FMiff_lemma}
For a Markov operator $M \colon {\rm C}( \mathbb{Y}) \to  {\rm C}( \mathbb{Y})$,
we have that $F_{{\rm meas}} (M^*) = \mathfrak{M}_{1}({\mathbb{Y}}) $  if and only if  $F_{{\rm pt}} ^0(M^*)=\mathbb{Y}$.
\end{lemma}

\begin{proof}
It is easy to check that if  $F_{{\rm meas}} (M^*) = \mathfrak{M}_{1}({\mathbb{Y}}) $ then $F_{{\rm pt}} ^0(M^*)=\mathbb{Y}$.
Conversely, suppose that  $F_{{\rm pt}} ^0(M^*)=\mathbb{Y}$.
If there exists some $\mu \in \mathfrak{M}_{1}({\mathbb{Y}}) \setminus F_{{\rm meas}} (M^*) $,
then there exist  $\varepsilon >0$ such that for any $ j \in \nn$, there exist some $n_j \in \nn$ and some $ \mu_j \in \mathfrak{M}_{1}({\mathbb{Y}})$ such that $d_0(\mu ,\mu_j) \leq {j}^{-1}$ and $d_0((M^*)^{n_j} (\mu) ,(M^*)^{n_j}(\mu_j)) \geq \epsilon$.
Fix some $N \in \nn$ so that $\sum_{n=N+1}^{\infty} {2^{-n}} < \varepsilon /2$ holds and set $$\eta =\frac{\varepsilon /N}{1 - \varepsilon /N}. $$
Then there exists $\phi = \phi_k \in {\rm C}( \mathbb{Y})$ such that $| (M^*)^{n_j} (\mu)(\phi) - (M^*)^{n_j}(\mu_j)(\phi) | \geq \eta$ holds for infinitely many  $ j \in \nn$.
By  Lemma \ref{fpt0_lemma} and the assumption that $F_{{\rm pt}} ^0(M^*)=\mathbb{Y}$, 
the family $\{ M^n \phi \}_{n \in \nn}$ is equicontinuous on $\mathbb{Y}$.
According to the Arzel\'a-Ascoli theorem, we can assume that $\{ M^{n_j} \phi \}_{j \in \nn}$ converges to some $\psi \in {\rm C}(\mathbb{Y})$ uniformly on $\mathbb{Y}$.
Thus, for sufficiently large $j \in \nn$, we have
\begin{align*}
&\quad | (M^*)^{n_j} (\mu)(\phi) - (M^*)^{n_j}(\mu_j)(\phi) | \\
& \leq |\mu (M^{n_j}\phi) - \mu (\psi) | +|  \mu (\psi)  -\mu_j (\psi)| +|\mu_j (\psi) -   \mu_j (M^{n_j}\phi) | \\
& < \frac{\eta}{3} + \frac{\eta}{3}  + \ \frac{\eta}{3}  = \eta,
\end{align*}
which leads to a contradiction.
\end{proof}

\section{Settings of Markov random dynamical systems}\label{setting_sec}
In this section, we consider a  GDMS $S_\tau$ that is induced by a given family $\tau$ of measures;
we define probability measures on the space of infinite product of $\cm (Y)$ and define a Markov operator $M_\tau$ induced by $\tau$.
Furthermore,  we show that almost every random Julia set  is a null set if the kernel Julia set is empty (Proposition \ref{pathwisenull_prop}).

\begin{setting}\label{ms_setting}
Let $Y$ be a compact metric space and let $m \in \nn$.
Suppose that $m^2$ measures $ (\tau_{ij})_{i, j =1, \dots, m}$ on $\cm (Y)$ satisfy  $ \sum_{ j =1}^m \tau_{i j}(\cm (Y)) = 1$ for all $i =1, \dots, m$.
For a given  $\tau = (\tau_{ij})_{i, j =1, \dots, m}$, we consider the Markov chain on $Y \times \{1, \dots, m\}$ whose transition probability from $(y, i ) \in Y \times \{1, \dots, m\}$ to $ B \times \{ j \}$ is $ \tau_{ij}(\{ f \in {\rm CM}(Y) ; f(y) \in B\}) $, where $B$ is a Borel subset of $Y$  and $ j \in  \{1, \dots, m\}$.
We call this Markov chain the {\it Markov random dynamical system} (MRDS for short) induced by $\tau$.
\end{setting}

\begin{definition}\label{mstau_rem}
\begin{enumerate}
\item[(I)] When a family $\tau$ of measures is given as in Setting \ref{ms_setting}, we define the GDMS $S_\tau$ in  the following way.
Define the vertex set as $V:= \{ 1,2, \dots , m\} $ and the edge set as $$E :=\{ (i ,j) \in V \times V ; \, \tau _{ij}(\cm (Y))  > 0 \}. $$
Also, for each $e=(i,j) \in E$, we define $\Gamma_e :=  \supp \tau _{ij} $. 
Set $S_{\tau} :=(V,E,(\Gamma_e)_{e \in E})$, which we call the GDMS induced by $\tau$.
We define $i : E \to V$ (resp. $t : E \to V$) as the projection to the  first (resp. second) coordinate and we call  $i(e)$ (resp. $t(e)$) the initial (resp. terminal) vertex of $e \in E$.
\item[(II)] We say that $\tau = (\tau_{ij})_{i, j =1, \dots, m}$ is {\it irreducible} if $S_{\tau}$ is irreducible.
\end{enumerate}
\end{definition}

In the following, let $\tau$ be a family of measures as in Setting \ref{ms_setting}.
Set $\mathbb{Y} := Y \times V$.
We can define a metric on $\mathbb{Y}$ using the metric on $Y$ and regard the compact metric space  $\mathbb{Y}$ as   $m$ copies of $Y$: $\mathbb{Y} \cong \bigsqcup_{V} Y$. 

\begin{definition}\label{titil_def}
We define Borel probability measures $\tilde{\tau}_i \, (i \in V)$ on $X_i(S_\tau)$ as follows.
For $N$ Borel sets $A_n \,(n=1,\dots,N)$ of  $\cm(Y)$ and for  $(e_1, \dots , e_N) \in E^N$,
set $A'_n = A_n \times \{ e_n \} $.
We define the measure $\tilde{\tau}_i$ on $(\cm(Y) \times E )^{\nn}$ so that 
 \begin{align*}
 &\tilde{\tau}_i \left( A'_1 \times \cdots \times A'_N \times \prod_{N+1}^{\infty} (\cm(Y) \times E) \right) \\
 =& \begin{cases}
     \tau_{e_1}(A_1) \cdots \tau_{e_N}(A_N), & \text{if } (e_1, \dots , e_N) \text{ is admissible  with }  i(e_1)=i \\
     0, & \text{otherwise}
   \end{cases}
   \end{align*}
for each $i \in V$.
Note that $\supp \tilde{\tau}_i =X_i(S_\tau)$.
\end{definition}

\begin{lemma}\label{irred_lemma}
We set $p_{ij}=\tau_{ij}(\cm (Y) )$ and set $P = (p_{ij})_{i,j =1, \dots,m}$. 
Then the following statements hold.
\begin{enumerate}
\item A GDMS $S_\tau$ is irreducible if and only if the matrix $P$ is irreducible.
\item If $S_\tau$ is irreducible, then there exists the unique vecter $p=(p_1,\dots,p_m)$  such that $pP=p$, $\sum_{i \in V} p_i =1$ and $p_i > 0$ for all $ i\in V$.
\item Assume  $S_\tau$ is irreducible and define the probability measure $\tilde{\tau}$ on $(\cm(Y) \times E)^{\nn}$ as $\tilde{\tau}=\sum_{i=1}^m p_i \tilde{\tau}_i$, where the vector $p$ is as above.
Then $\supp \tilde{\tau} =X(S_\tau)$ and $\tilde{\tau} $ is an invariant probability measure with respect to the shift map on $X(S_\tau)$.
\end{enumerate}
\end{lemma}

\begin{proof}
We show (iii).
For a Borel set $\tilde{A}$ of $(\cm(Y) \times E )^{\nn}$, we prove $\tilde{\tau}(\sigma^{-1} (\tilde{A}))=\tilde{\tau}(\tilde{A})$.
We may assume 
 \[ \tilde{A}=A'_1 \times \cdots \times A'_N \times \prod_{N+1}^{\infty} (\cm(Y) \times E) , \quad A'_n = A_n \times \{ e_n \}. \] 
If the word $(e_1, \dots ,e_N)$ is not admissible, then $\tilde{\tau}(\sigma^{-1} (\tilde{A}))=0= \tilde{\tau}(\tilde{A})$.
If  $(e_1, \dots ,e_N)$ is admissible, then
 \[ \sigma^{-1}(\tilde{A}) =(\cm (Y) \times E ) \times \tilde{A} = \bigsqcup_{i \in V} \bigsqcup_{ i(e) =i} (\cm (Y) \times \{ e \} ) \times \tilde{A} \]
and hence
 \begin{align*}
 \tilde{\tau}(\sigma^{-1} (\tilde{A}))& =\sum_{i \in V}p_{i} p_{i i(e_1)} \tau_{e_1}(A_1 )\cdots \tau_{e_N}(A_N )  
  =p_{i(e_1)} \tau_{e_1}(A_1 )\cdots \tau_{e_N}(A_N ) 
=\tilde{\tau}(\tilde{A}).
 \end{align*}
\end{proof}

\begin{definition}\label{def-Mtau}
For $\tau = (\tau_{ij} )_{i,j \in V}$, we define the transition operator $M_{\tau}$ of $\tau$ as follows.
\[ M_{\tau} \phi (y,i) := \sum_{j \in V}  \int_{\Gamma_{ij}} \phi (\gamma (y),j) \,  {\rm d} \tau_{ij} (\gamma) ,  \quad (y, i ) \in \mathbb{Y}.\]
Here, $\phi$ is a complex-valued Borel measurable function on  $\mathbb{Y}$.
\end{definition}

\begin{rem}
For the transition operator  $M_{\tau}$,  the following statements hold.
\begin{enumerate}
\item If $\phi \in  {\rm C}( \mathbb{Y})$, then $M_\tau \phi \in  {\rm C}( \mathbb{Y})$.
\item  The transition operator  $M_{\tau}$ is a Markov operator on ${\rm C}( \mathbb{Y})$ in the sense of subsection \ref{mrkvop_subsec} (see Definition \ref{mrkvop_def}).
\end{enumerate}
\end{rem}

\begin{lemma}\label{transop_lemma}
If $ (y, i ) \in \mathbb{Y}, n \in \nn$ and $\phi \in {\rm C} (\mathbb{Y})$, then
\[ (M_{\tau} ^n \phi) (y,i) =  \int_{X_i(S_\tau)} \phi (\xi_{n,1} (y,i)) \,  {\rm d} \tilde{\tau}_{i} (\xi) .\]
\end{lemma}

For the meaning of $\xi_{n,1} (y,i)$,  see Notation \ref{xi_nota}.

\begin{proof}
We use induction on $n$.
If the statement holds for  $n=N$, then
\begin{align*}
(M_{\tau} ^{N+1} \phi) (y,i) 
&=\sum_{i_0 \in V}  \int_{\Gamma_{i i_0}}(M_{\tau} ^{N}  \phi) (\gamma_0 (y),i_0) \,  {\rm d} \tau_{i i_0} (\gamma_0) \\
&=\sum_{i_0 \in V}  \int_{\Gamma_{i i_0}} \int_{X_{i_0}(S_\tau)} \phi (\xi_{N,1} (\gamma_0(y),i_0)) \,  {\rm d} \tilde{\tau}_{i_0} (\xi){\rm d} \tau_{i i_0} (\gamma_0) .
\end{align*}
Set $ \xi_0=(\gamma_0,e_0), e_0=(i\, i_0)$ and $\xi =(\gamma_n,e_n)_{n \in \nn}, e=(e_n)_{n=1}^N$.
If $\phi = \1_{B \times \{ j \}} $, then 
\begin{align*}
&(M_{\tau} ^{N+1} \phi) (y,i) \\
=&\sum_{i_0 \in V} \tau_{i\,i_0 }\otimes   \tilde{\tau}_{i_0} ( \{ (\xi_0,\xi)\, ;\,  i_0=i(e) , t(e) =j \text{ and } \gamma_{N} \circ \dots \circ \gamma_{0} (y) \in B \}) \\
=& \sum_{e'} \tau_{e_0} \otimes \tau_{e_1} \otimes \cdots \otimes \tau_{e_N} (\{(\gamma_n)_{n=0}^N \in \cm(Y)^{N+1} \ ;\, \gamma_{N} \circ \dots \circ \gamma_{0} (y) \in B   \}) \\
=& \int_{X_i(S_\tau)} \phi (\xi_{N+1,1} (y,i)) \,  {\rm d} \tilde{\tau}_{i} (\xi).
\end{align*}
Here, the summation is taken over all admissible words  $e' =(e_n)_{n=0}^{N}$ with initial vertex $i$, terminal vertex $j$ and length $N+1$.
This completes the proof since any continuous function $\phi$ can be approximated by simple functions.
\end{proof}

\begin{lemma}\label{measure0_lemma}
If  $\tilde{\tau}_i ( \{ \xi \in X_i(S_\tau) ; y \in J_\xi \}  ) =0$ holds for  $(y,i) \in \mathbb{Y}$, then $ (y, i ) \in F_{\rm pt}^0 (M_\tau^*)$.
\end{lemma}

\begin{proof}
By assumption, we have  $ y \in F_\xi$ for $\tilde{\tau}_{i}$-almost every  $\xi \in X_{i}(S_{\tau }) .$
For such $\xi =(\gamma_n , e_n )_{n \in \nn}$, we have $\lim_{\eta \to 0} \sup_{n \in \nn} {\rm diam} (\gamma_{n,1} B(y , \eta) )=0 $ by Lemma \ref{limsupdiam_lemma}.
For any $\phi \in {\rm C}(\mathbb{Y})$ and any $\varepsilon >0$,
the function $\phi$ is uniformly continuous on the compact space $\mathbb{Y}$.
Thus there exists  $\delta_1 > 0$ such that for any  $z_1, z_2 \in \mathbb{Y}$ with $d(z_1, z_2 ) < \delta_1$, we have $| \phi(z_1) - \phi (z_2) | < \varepsilon$.
By Egoroff's theorem, there exists a Borel set $B \subset X_i (S_\tau)$ with $\tilde{\tau}_i(B^c) = \tilde{\tau}_i(X_i(S_\tau) \setminus B) < \varepsilon$ satisfying the following property;
there exists  $\eta_0 >0$ such that $\sup_{n \in \nn} {\rm diam} (\gamma_{n,1} B(y , \eta_0) ) < \delta_1$  for any $\xi =(\gamma_n , e_n )_{n \in \nn} \in B$.
Hence, for any  $z_1=(y_1,i)$  whose distance from $z=(y,i)$ is less than $\eta_0$, we have
\begin{align*}
| (M_\tau^n \phi) (z) - (M_\tau^n \phi) (z_1) | 
& \leq \int_{X_i(S_\tau)}| \phi (\xi_{n,1} (z)) -   \phi (\xi_{n,1} (z_1))| \,  {\rm d} \tilde{\tau}_{i} (\xi)  \\
& = \int_B + \int_{B^c} \leq \varepsilon \tilde{\tau}_i(B) + 2 \| \phi \| \tilde{\tau}_i(B^c) 
 \leq \varepsilon ( 1 +  2 \| \phi \|).
\end{align*}
By Lemma \ref{fpt0_lemma}, we have $z= (y, i ) \in F_{\rm pt}^0 (M_\tau^*)$.
\end{proof}

\begin{cor}
Let $\lambda$ be a Borel finite measure on  $\mathbb{Y}$.
If $\lambda (J_\xi) =0$  for all $i \in V$ and for $\tilde{\tau}_i $ -a.e. $\xi \in X_i(S_\tau)$, then $\lambda(\mathbb{Y} \setminus F_{\rm pt}^0 (M_\tau ^*) )=0$. 
\end{cor}
 
 \begin{proof}
The statement  follows easily from Lemma \ref{measure0_lemma} and Fubini's theorem.
 \end{proof}

\begin{lemma}\label{lem-M}
Let $(U_j)_{j \in V}$ be a forward $S_\tau$-invariant family such that  each $U_j$ is a non-empty open subset of $Y$.
Set $L_{{\rm {\rm ker}},j} = \bigcap_{k \in V\colon H_j^k (S_\tau)  \neq \emptyset} \bigcap_{h \in H_j^k (S_\tau)} h^{-1} (Y \setminus U_k)$ for  $j \in V$.
Also,  for $(y,i) \in \mathbb{Y}$, we set  $$E=\{ (\gamma_n , e_n )_{n \in \nn} \in  X_i(S_\tau) ; y \in \bigcap_{n\in \nn} \gamma_{n,1}^{-1} (Y \setminus U_{t(e_n)}) \} .$$
Then $  d(\gamma_{n,1} (y) , {L}_{{\rm {\rm ker}}, t(e_n)}) \to 0 \, (n \to \infty)$ for $\tilde{\tau}_i $ -a.e. $(\gamma_n , e_n )_{n \in \nn}  \in E$, where $d(a , \emptyset) :=\infty \ (a \in Y)$.
\end{lemma} 

 \begin{proof}
Let $z=(y,i) \in \mathbb{Y}$,  $\mathbb{U} := \bigcup_{j \in V } U_j \times \{ j \}$ and $\mathbb{L}_{\rm ker} := \bigcup_{j \in V } L_{{\rm {\rm ker}},j} \times \{ j \}$.
For any $\varepsilon  > 0$ and for any $n \in \nn$, we set
 \begin{align*}
A(\varepsilon , n)  &:= \{ \xi \in E \, ;\, \xi_{n,1}(z) \notin \mathbb{U} \cup B(\mathbb{L}_{\rm ker},\varepsilon ) \} ,\\
C(\varepsilon ) & := \{ \xi \in E \, ;\, \exists N \in \nn \text{ such that } \xi_{n,1}(z) \in B(\mathbb{L}_{\rm ker},\varepsilon ) \text{ for any } n \geq N \}.
 \end{align*}
Here, 
$B (\mathbb{L}_{\rm ker}, \varepsilon) = \{ y \in \mathbb{Y} ; \, d(y, \mathbb{L}_{\rm ker} ) < \varepsilon \}  $
 and we set
 $B(\mathbb{L}_{\rm ker},\varepsilon ) = \emptyset$ if $\mathbb{L}_{\rm ker} = \emptyset$.
We  prove that  $\tilde{\tau}_i( E \setminus C(\varepsilon ) ) =0$ for any  $\varepsilon >0$.
For this purpose, fix a small $\varepsilon >0$.
It suffices to show $\sum_{n \in \nn} \tilde{\tau}_i(A(\varepsilon , n) ) < \infty$. 
For, since $ E \setminus C(\varepsilon ) =\limsup_{n \to \infty} A(\varepsilon , n)  $,  the statement follows by combining these with the Borel-Cantelli lemma.
 
In order to show $\sum_{n \in \nn} \tilde{\tau}_i(A(\varepsilon , n) ) < \infty$, we set
 $ \mathbb{K} := \mathbb{Y} \setminus (   \mathbb{U} \cup B(\mathbb{L}_{\rm ker},\varepsilon ) )$.
Then there exist subsets $K_j \subset  \bigcup_{k \in V\colon  H_j^k (S_\tau) \neq \emptyset} \bigcup_{h \in H_j^k (S_\tau)} h^{-1}(U_k) ,\, j\in V, $ such that $\mathbb{K} =  \bigcup_{j \in V } K_j \times \{ j \}$.
Since $\mathbb{K}$ is compact,
there exist finitely many open sets $W_q $ in $\mathbb{Y}$ ($q =1, \dots,p$) and finitely many admissible sequences $g_q \in ( \cm(Y) \times E )^{l_q}$ ($q =1, \dots,p$) such that $\mathbb{K} \subset \bigcup_{q=1}^p W_q$ and $g_q (W_q) \subset \mathbb{U}$.
Note that
we may assume there exists $l \in \nn$ such that $l = l_q$ for all $q =1, \dots,p$ since  $(U_j)_{j \in V}$ is forward  $S_\tau$-invariant.
Then, for each $q =1, \dots,p$, there exists an open neighborhood $O_q \subset  ( \cm(Y) \times E )^l$ of $g_q$ such that  $g(W_q) \subset \mathbb{U}$ for all $g \in O_q$.
We put $\tilde{O}_q := O_q \times \prod_{l +1} ^\infty( \cm(Y) \times E )$ and $\delta :=\min_{q=1, \dots , p} \tilde{\tau}_{i_q}  (\tilde{O}_q)   > 0$, where  $i_q \in V$ is the initial vertex of $g_q$.
 
For each $k \geq 0$ and $r= 0, \dots , l-1$, we set
 \begin{align*}
 I(k,r) := \{ \xi \in X_i (S_\tau) \, ;\, \xi_{k l +r ,1}(z) \in \mathbb{K} \} \text{ and }
 H(k,r) := \{ \xi \in I(k,r) \, ;\, \xi_{(k+1) l +r ,1}(z) \in \mathbb{U} \}.
 \end{align*}
Here, $I(0,0):=\emptyset$.
If $k \neq k'$, then $ H(k,r) \cap H(k',r) = \emptyset$.
Since $\mathbb{K} \subset \bigcup_{q=1}^p W_q$, there exist  $s$ Borel sets $B_1 ,\dots , B_s$ on $\mathbb{Y}$ for some $s \in \nn$  with the following property; 
$\mathbb{K} = \bigsqcup_{t =1}^s B_t$, where $\bigsqcup$ denotes the disjoint union, and for each $t=1, \dots, s$  there exists  $q(t) \in \{ 1, \dots ,r \}$ 
such that $B_t \subset W_{q(t)}$.
Then, we have
 \begin{align*}
 \tilde{\tau}_i(H(k,r))
  &=  \sum_{t=1}^s \tilde{\tau}_i (\{ \xi \in X_i (S_\tau)\, ; \,  \xi_{k l +r ,1}(z) \in B_t  , \xi_{(k+1) l +r ,1}(z) \in \mathbb{U}  \}) \\
 &\geq \sum_{t=1}^s \tilde{\tau}_i (\{ \xi \in X_i (S_\tau)\, ; \,  \xi_{k l +r ,1}(z) \in B_t  , \xi_{(k+1) l +r ,k l +r +1} \in O_{q(t)}  \}) \\
& \geq \sum_{t=1}^s \tilde{\tau}_i (\{ \xi \in X_i (S_\tau)\, ; \,  \xi_{k l +r ,1}(z) \in B_t  \} ) \delta 
=  \tilde{\tau}_i(I(k,r)) \delta
 \end{align*}
and hence
\begin{align*}
1 \geq 
 \tilde{\tau}_i( \bigcup_{k \geq 0} H(k,r)) = \sum_{k=0}^\infty 
 \tilde{\tau}_i(H(k,r)) \geq \delta  \sum_{k=0}^\infty 
 \tilde{\tau}_i(I(k,r)).
\end{align*}
It follows that $\sum_{n \in \nn} \tilde{\tau}_i(A(\varepsilon , n) ) \leq l / \delta  < \infty$.
\end{proof}

The following proposition is one of the main results of this paper.
The statement means that  almost surely the random Julia set is of measure-zero and the averaged system is stable
if the kernel Julia set is empty.

\begin{prop}\label{pathwisenull_prop}
Let $\lambda $ be a Borel finite measure on $Y$. 
Suppose that $\mathbb{J}_{{\rm {\rm ker}}}(S_\tau) = \emptyset$ and  $\Gamma_e \subset \ocm(Y)$ for all  $e \in E$.
Then,  $F_{\rm meas} (M_\tau^* ) =\mathfrak{M}_1 (\mathbb{Y})$ 
and 
 $ \lambda(J_\xi) = 0$ holds 
for any $i \in V$ and for $\tilde{\tau}_i $ -a.e. $\xi \in X_i(S_\tau)$.
\end{prop}

 \begin{proof}
Note that  the Fatou set  $F_j (S_\tau)$ at   each $j \in V$ is not empty since $\mathbb{J}_{{\rm {\rm ker}}}(S_\tau) = \emptyset$.
By Lemma \ref{inv_lemma} the family  $(F_j(S_\tau))_{j \in V}$ of Fatou sets  is forward $S_\tau$-invariant. 
Hence we can apply Lemma \ref{lem-M} with  $ U_j:=F_j (S_\tau)$. 
Therefore, we have $$\tilde{\tau}_i \{ (\gamma_n , e_n )_{n \in \nn} \in  X_i(S_\tau); y \in \bigcap_{n\in \nn} \gamma_{n,1}^{-1} \left(Y \setminus F_{t(e_n)}(S_\tau) \right) \} =0$$ for all $(y, i) \in \mathbb{Y}$.
By Lemma \ref{limsupdiam_lemma}, it follows that $ \tilde{\tau}_i (\{ \xi \in  X_i(S_\tau); y \in J_\xi \} ) =0$.
By virtue of Fubini's theorem, we have $\lambda(J_\xi) = 0$ for $\tilde{\tau}_i $ -a.e. $\xi \in X_i(S_\tau)$.
Furthermore, by Lemma \ref{measure0_lemma},  we know  $ (y, i ) \in F_{\rm pt}^0 (M_\tau^*)$ for any $(y, i) \in \mathbb{Y}$.
Lemma \ref{FMiff_lemma} implies $F_{{\rm meas}} (M_\tau^*) = \mathfrak{M}_{1}({\mathbb{Y}}) $.
 \end{proof}
\section{Rational MRDSs on $\rs$}\label{ratMar_sec}
In this section, we focus on  holomorphic dynamical systems on the Riemann sphere $\rs$.
We denote by $\rat$ the space of all non-constant holomorphic maps from $\rs$ to itself with the topology of  uniform convergence or the compact-open topology.
Recall that each element $f$ of $\rat$ can be expressed as the quotient $p(z) / q(z)$ of two polynomials without common zeros and the degree of $f$  is defined by the maximum of the degrees of $p$ and $q$. 
We denote by $\poly$ the subspace of $\rat$ consisting of all polynomial  maps of degree two or more.
We consider {\it rational} GDMSs or {\it polynomial} GDMSs as in  Definition \ref{def-RMRDS}.

In subsection \ref{julia_sec},  we discuss the Julia sets of rational GDMSs and show some fundamental properties.
These discussions are the generalization of those of rational semigroups (see \cite{stank12}).
Moreover, we show some sufficient conditions for the kernel Julia sets to be empty.
In subsection \ref{probfunc_section}, we focus on a polynomial MRDS  
and the function $\mathbb{T}_{\infty,\tau}$ which represents the probability of tending to $\infty$.
We show that the function $\mathbb{T}_{\infty,\tau}$ is continuous on the whole space and varies precisely on the Julia set of associated system under certain conditions.
 
\subsection{Julia sets}\label{julia_sec}
\begin{definition}\label{def-RMRDS}
We say that  $S = ( V, E, (\Gamma_e)_{e \in E} )$  is a rational (resp. polynomial) GDMS on $\rs$ if $\Gamma_e \subset \rat$  (resp. $\Gamma_e \subset \poly$ ) for all $e \in E$.
\end{definition}

Let $S = ( V, E, (\Gamma_e)_{e \in E} )$ be a rational GDMS.
Recall that the Julia set $J_i(S)$ of $S$ at the vertex $i \in V$ is equal to the Julia set of the rational semigroup $H_i^i(S)$ if $S$ is irreducible  (see Remark \ref{sysJulia=semiJulia_rem}).
It is well known that the Julia set $J(G)$ of a rational semigroup $G$ is equal to the closure of the set of repelling fixed points of elements of $G$ if $J(G)$ contains at least three points.
For this reason, we introduce the following important definition.

\begin{definition}
A rational GDMS $S= (V, E,(\Gamma_e)_{e \in E})$  is said to be non-elementary if the Julia set  $J_i(S)$ at $i$ contains  at least three points for all $i \in V$.
\end{definition}

Consequently, we obtain a characterization of the Julia set of a rational GDMS.

\begin{cor}\label{cor-density}
If a rational GDMS $S= (V, E,(\Gamma_e)_{e \in E})$  is non-elementary and irreducible, then 
$$J_i(S) =\overline{ \bigcup_{ h \in H_i^i(S)} \{ \text{repelling fixed points of }h \} }= \overline{\bigcup_{\xi \in X_i(S)} J_{\xi}}$$
 for all $i \in V$.
Here, a fixed point $z_0$ of $h$ is said to be repelling if the modulus of multiplier of $h$ at $z_0$  is greater than $1$.
\end{cor}

Here are some basic properties of the Julia set.
Although some claims can be directly proved by using the theory of rational semigroups, we do not use it in order not to require preliminary knowledge. 

\begin{lemma}\label{lem-backwrdMinimal}
Let $(L_i)_{i \in  V}$ be a family  that is backward $S$-invariant and suppose that each  $L_i$ is a compact set which contains at least three points (\cite{stank12}).
Then $(J_i(S))_{i \in  V} \subset  (L_i)_{i \in  V}$.
\end{lemma}

\begin{proof}
Since $(L_i)_{i \in  V}$  is backward $S$-invariant, it follows that $H_i^j(S)(\rs \setminus L_i) \subset \rs \setminus L_j$ for all  $ i, j \in V$.
If $\# L_j \geq 3$ for each $j \in V$, then  $\rs \setminus L_i \subset F_i(S)$  for each $i \in V$ by Montel's theorem.
\end{proof}

\begin{definition}\label{def-exceptional}
A point $z$ is called an {\it exceptional point}  of $S$ at the vertex $i \in V$ if $\# (H_i^i(S))^{-1}(z) < 3 $, where $(H_i^i(S))^{-1}(z) = \bigcup_{h \in H_i^i(S)} h^{-1}(z)$.
We define  $\mathcal{E}_i (S)$ as the set of all exceptional points $z$ of $S$ at $i \in V$.
\end{definition}

\begin{lemma}\label{lem-backwrdOrbit}
Let $S= (V, E,(\Gamma_e)_{e \in E})$  be an irreducible rational GDMS.
Let $j \in V$.
Then we have the following statements.

\begin{enumerate}
\item If $z \notin \mathcal{E}_j(S)$, then  $J_i(S) \subset \overline{(H_i^j(S))^{-1}(z)}$ for all  $i \in V$. 
\item If $z \in J_j(S) \setminus  \mathcal{E}_j(S)$, then  $J_i(S) = \overline{(H_i^j(S))^{-1}(z)}$ for all $i \in V$. 
\end{enumerate}
\end{lemma}

\begin{proof}
Set $L_i := \overline{(H_i^j(S))^{-1}(z)}$.
Then $(L_i)_{i \in V}$ is  backward $S$-invariant and each $L_i$ contains at least three points since $S$ is irreducible. 
Lemma \ref{lem-backwrdMinimal} implies (i).
Combining (i) and Lemma \ref{inv_lemma} implies (ii).
\end{proof}

\begin{rem}
Lemma \ref{lem-backwrdOrbit} suggests an algorithm for computing pictures of the Julia set.
Figure \ref{ex-Julia} is drawn in this manner.
\end{rem}

\begin{lemma}\label{perfect_lemma}
If a rational GDMS $S= ( V, E, (\Gamma_e)_{e \in E} )$ is non-elementary, then each Julia set $J_i(S)$ is a perfect set.
\end{lemma}

\begin{proof}
Suppose $J_i(S)$  has an isolated point  $z_0 $.
Then there exists an open neighborhood $U$ of $z_0$ in $\rs$ such that $U \cap J_i(S) = \{ z_0 \}$.
Set $U^* := U \setminus \{ z_0 \} $.
We, therefore, have that  $ U^* \subset F_i(S)$ and 
 $H_i^j(S) (U^*) \subset F_j(S)$ for all  $j \in V$.
By assumption, $\rs \setminus F_j(S)$ contains at least three points for all $j \in V$.
Thus, by  the strengthened Montel's theorem \cite[p203]{cara} 
it follows that  $H_i^j(S)$ is normal on the whole $U$.
This contradicts that $z_0 \in J_i(S)$.
\end{proof}

\begin{lemma}\label{exceptional_lemma}
If a rational GDMS $S= ( V, E, (\Gamma_e)_{e \in E} )$ is non-elementary and irreducible, then $\# \mathcal{E}_i (S) <3$ for all $i \in V$.  
\end{lemma}

\begin{proof}
The proof is by contradiction. 
Suppose that there exists $k \in V$ such that $\mathcal{E}_k (S)$ has three distinct points $a$, $b$ and $c$.
For each $i \in V$, we set $L_i:= \overline{(H_i^k(S))^{-1}(\{a,b,c \} )}$.
Then $(L_i)_{i \in V}$ is backward $S$-invariant and each $L_i$ contains at least three points.
By Lemma \ref{lem-backwrdMinimal}, we have $J_k(S) \subset L_k$.
However, this contradicts Lemma \ref{perfect_lemma} since $L_k$ is finite.
\end{proof}

\begin{prop}\label{prop-intEmpty}
Let  $S = ( V, E, (\Gamma_e)_{e \in E} )$ be an irreducible rational GDMS such that  $\Gamma_e$ is a finite set  for each $e\in E$.
If $S$ satisfies the backward separating condition,
then either   ${\rm int} (J_{i}(S)) = \emptyset $ for all $i \in V$, or  $J_{i}(S) = \rs $ for all $i \in V$.
\end{prop}

\begin{proof}
We assume that ${\rm int} (J_{i}(S)) \neq \emptyset$  for some $i \in V$  and prove that $J_i(S) = \rs$.
Let  $U$ be a connected open subset of ${\rm int} (J_{i}(S))$. 
By the backward separating condition and Proposition \ref{selfsim_prop},
there uniquely exist  $e_1 \in E$ and   $f_1 \in \Gamma_{e_1}$ such that $i=i(e_1)$ and $U \subset f_1^{-1}(J_{t(e_1)}(S))$.
Furthermore, for  $e \in E$ with $i(e)=i$ and  $f \in \Gamma_e$, if $e \neq e_1$ or $f \neq f_1$, then $U \cap f^{-1}(J_{t(e)} (S)) = \emptyset$.
Inductively, there uniquely  exist $e_n \in E$ and $f_n \in \Gamma_{e_n}$ such that  $t(e_n)=i(e_{n +1})$ and $f_n \circ \cdots \circ f_1(U) \subset J_{t(e_n)} (S)$ for any $n \in \nn$.

By Lemma \ref{sysJulia=semiJulia_lemma}, we have  $U \subset J_i(S) =J(H_i^i(S))$ and hence there exists a sequence $\{ h_n \}_{n\in \nn}  \subset H_i^i(S) $ that contains no subsequence which converges locally uniformly on $U$.
It follows from Montel's theorem that there exists a subsequence  $\{ h_{n(k)} \}$ such that $ h_{n(k)} \in \{ f_n \circ \cdots \circ f_1 \}_n$  for all $k \in \nn$.
Thus we have $h_{n(k)}(U) \subset J_i(S)$ for all $k \in \nn$, and hence $J_i(S) = \rs$ by Montel's theorem again.
\end{proof}

Now we investigate the kernel Julia sets of rational GDMSs.

\begin{lemma}\label{intJker_lemma}
Let $S= (V, E,(\Gamma_e)_{e \in E})$  be an irreducible rational GDMS.
If ${\rm int}( J_{{\rm ker} ,j}(S) ) \neq \emptyset$  for some  $j \in V$, then $J_{{\rm ker} ,i}(S) = \rs$ for all $i \in V$.
\end{lemma}

\begin{proof}
The proof is by contradiction. 
It suffices to show that  $J_{i} (S)= \rs$ for all  $i \in V$.
We assume that there exists  $i \in V$ such that $J_{\ker, i} (S) \neq \rs$.
Then $\# \rs \setminus J_{{\rm ker} ,i}(S) \geq 3$ and $h( {\rm int}( J_{{\rm ker} ,j}(S) ) ) \subset  J_{{\rm ker} ,i} (S) $ for all $h \in H_j^i(S)$ by Lemma \ref{kerjulia_lemma}.
It consequently follows that $H_j^i(S)$ is normal on  ${\rm int}( J_{{\rm ker} ,j}(S) ) $.
Now we fix some  $g \in H_j^i$ and hence $g \circ H_j^j(S) \subset H_j^i(S)$.
By Lemma \ref{equiconti_lemma} and the  Arzel\'a-Ascoli theorem, the family $H_j^j(S)$ is equicontinuous on ${\rm int}( J_{{\rm ker} ,j} (S)) $, so that ${\rm int}( J_{{\rm ker} ,j}(S) ) \subset F_j(S)$.
This contradicts the fact that $\emptyset \neq {\rm int}( J_{{\rm ker} ,j} (S)) \subset   J_j(S)$.
\end{proof}

\begin{definition}
Let $\Lambda$ be a connected finite-dimensional complex manifold.
Let $\{ g_\lambda : \rs \to \rs \}_{\lambda \in \Lambda}$ be a family of non-constant rational maps on $\rs$.
We say that $\{ g_\lambda \}_{\lambda \in \Lambda}$ is a holomorphic family of rational  maps (over $\Lambda$) if 
the associated map $\rs \times \Lambda \ni (z , \lambda ) \mapsto g_\lambda (z) \in \rs$ is holomorphic.
\end{definition}

\begin{prop}\label{holofami_lemma}
Suppose that an irreducible rational GDMS $S= (V, E,(\Gamma_e)_{e \in E})$  has  $e \in E$ with the following property.
For all $z \in J_{i(e)}(S)$, there exists a holomorphic family of rational maps $\{ g_\lambda  \}_{\lambda \in \Lambda} \subset \Gamma_e$ such that the map $\Theta :\Lambda \ni \lambda  \mapsto g_\lambda (z) \in \rs$ is non-constant.

If, in addition, $F_{i(e)}(S) \neq \emptyset $, then $ J_{{\rm ker} ,i}(S) = \emptyset$ for all $i \in V$.
\end{prop}

\begin{proof}
The proof is by contradiction. 
Suppose  there exists an element $z \in J_{{\rm ker} ,i(e)} (S)$.
Fix a holomorphic family  $\{ g_\lambda  \}_{\lambda \in \Lambda} \subset \Gamma_e$ such that the map $\Theta :\Lambda \ni \lambda  \mapsto g_\lambda (z) \in \rs$  is non-constant.
Then $\Theta ( \Lambda)$ is non-empty and open  in $\rs$ and $\Theta ( \Lambda) \subset J_{{\rm ker} ,t(e)}(S)$ by Lemma \ref{kerjulia_lemma}.
It follows from Lemma \ref{intJker_lemma} that  $J_{t(e)}(S) \supset J_{{\rm ker} ,t(e)}(S) = \rs$, and this contradicts the assumption  $F_{i(e)}(S) \neq \emptyset $. 
By  Lemma \ref{inv_lemma},  $ J_{{\rm ker} ,i}(S) = \emptyset$ for all $i \in V$.
\end{proof}
\begin{cor}\label{cor-kernelJuliaCond}
Let $S$ be an irreducible polynomial GDMS. 
Suppose that $\Gamma _{e}$ is a compact subset of Poly for each $e\in E$. 
\begin{enumerate}
\item If there exists   $e \in E$ such that  ${\rm int} (\Gamma _{e} ) \neq \emptyset$,
then  $J_{{\rm ker} ,i} (S)=\emptyset$ for all $i\in V$.
Here, the symbol ${\rm int}$ denotes the set of all interior points in $\poly$.
\item If there exists $e \in E$,  $f\in \poly$ and a non-empty open set $U$ in $\Bbb{C}$ such that 
$\{ f+c \, ;\, c\in U\} \subset \Gamma _{e}$, then $J_{\ker ,i} (S) =\emptyset $ for all $i\in V$.
\end{enumerate}
\end{cor}

\begin{proof}
Since $\Gamma _{e}$ is a compact
subset of Poly for each $e\in E$, we have $\infty \in F_{i}(S)$ for each $i\in
V.$  Combining this with  Proposition 4.13, the statements (i) and (ii)
of our corollary hold. 
\end{proof}

%
%
%
%
%
\subsection{Probability tending to $\infty$}\label{probfunc_section}
In this subsection, we investigate polynomial MRDS induced by $\tau = (\tau_{ij})_{i,j =1, \dots,m}$ and its associated GDMS  $S_\tau = (V, E,(\Gamma_e)_{e \in E})$ such that $\Gamma_e$ is a compact subset of  $\poly$  for each $e \in E$.
For the definition of $S_\tau$, see Setting \ref{ms_setting} and Definition \ref{mstau_rem}.
Polynomial maps of degree $2$ or more have a common attracting fixed point at infinity, and hence 
some random orbits may tend to infinity.
We define the function $\mathbb{T}_{\infty, \tau} \colon \rs \times V \to [0, 1]$ which represents the probability of tending to infinity and give some sufficient conditions for $\mathbb{T}_{\infty, \tau}$ to be continuous on the whole space.
Moreover, we show that $\mathbb{T}_{\infty, \tau}$ is continuous on $\mathbb{Y}$  and varies precisely on the Julia set $\mathbb{J} (S_\tau)$ under certain conditions.
Recall that $\mathbb{Y}:=\rs \times V$.

\begin{definition}\label{Tinfty}
We define the function $\mathbb{T}_{\infty, \tau} \colon \mathbb{Y} \to [0, 1]$ by
\[\mathbb{T}_{\infty ,\tau}(z,i) := \tilde{\tau}_i (\{ \xi = (\gamma_n , e_n )_{n \in \nn} \in X_i(S_\tau) \, ; \, d(\gamma_{n,1}(z), \infty) \to 0 \, (n \to \infty) \} ) \]
for any point  $(z,i) \in \rs \times V$.
If $S_\tau$ is irreducible, we fix the vector  $p$ of Lemma \ref{irred_lemma} and define ${T}_{\infty, \tau} \colon \rs  \to [0, 1]$ by
\begin{align*}
 T_{\infty, \tau}(z):&= \sum_{i=1}^m p_i \mathbb{T}_{\infty ,\tau}(z,i) 
= \tilde{\tau} (\{ \xi = (\gamma_n , e_n )_{n \in \nn} \in X(S_\tau) \, ; \, d(\gamma_{n,1}(z), \infty) \to 0 \, (n \to \infty) \} ) .
 \end{align*}
\end{definition}

The function $ T_{\infty, \tau}$ is associated with  the following random dynamical systems. 
Fix an initial point $z \in \rs.$
We choose a vertex $i = 1, \dots, m$ with  probability $p_{i}.$ 
At the first step, we choose a vertex $i_{1} = 1, \dots, m$ with probability $\tau_{i i_{1} }(\poly)$ and choose a map $f_{1}$ according to the probability distribution $ \tau_{i i_{1}} / \tau_{i i_{1} }(\poly).$ 
Repeating this, we randomly choose  a map $f_{n}$ for each $n$-th step. 
Then the random orbit $f_{n} \circ \dots \circ f_{2} \circ f_{1} (z)$ tends to the point at infinity with probability $T_{\infty, \tau}(z).$  

We need the following lemma which can be easily shown.

\begin{lemma}\label{lem-attract}
Let $\Gamma$ be a compact subset of $\poly$.
Then there exists an open neighborhood  $U$ of  $\infty$ such that for all $\gamma =(\gamma_1, \gamma_2, \dots ) \in \Gamma^{\nn}$, we have $\gamma_{n,1}  \to \infty$ as $n \to \infty$ locally uniformly on $U$.
\end{lemma}

\begin{cor}\label{cor-cpt}
Let $S= (V, E,(\Gamma_e)_{e \in E})$  be a polynomial GDMS such that $\Gamma_e$ is a compact subset of  $\poly$  for all $e \in E$.
Then the Julia set $J_i(S)$ is a compact subset of $\Bbb{C}$ for all  $i \in V$.
\end{cor}

\begin{definition}\label{AK}
Let $ \xi = (\gamma_n , e_n )_{n \in \nn} \in X(S)$.
We denote by $A_\xi$ the set of all points $z$ such that $\gamma_{n,1}(z) \to \infty$ as  $n \to \infty$ and 
denote by $K_\xi$  the complement $\rs \setminus A_\xi$. 
\end{definition}

\begin{lemma}\label{prop-JEquiv}
Let $S= (V, E,(\Gamma_e)_{e \in E})$  be a polynomial GDMS such that $\Gamma_e$ is a compact subset of  $\poly$  for all $e \in E$.
Then the set $A_\xi$ is a  non-empty open set and $J_\xi =\partial K_\xi= \partial A_\xi$  for each $ \xi = (\gamma_n , e_n )_{n \in \nn} \in X(S)$.
\end{lemma}

\begin{proof}
Set $\Gamma :=\bigcup_{e \in E} \Gamma_e$.
We apply Lemma \ref{lem-attract} and fix the open neighborhood $U$ of $\infty$ in Lemma \ref{lem-attract}.
It follows easily that  $A_\xi = \bigcup_{n \in \nn} \gamma_{n,1}^{-1} (U )$ and hence $A_\xi$ is a non-empty open set.
For any open set $W$ of $\rs$ which meets $\partial K_\xi= \partial A_\xi$, the family $\{ \gamma_{n,1} \}_{n \in \nn}$  is not equicontinuous on $W$.
Thus $\partial K_\xi \subset J_\xi$.
Conversely, since  $\gamma_{n,1}(z) \to \infty$ as  $n \to \infty$ locally uniformly on $A_\xi$, we have $A_\xi \subset F_\xi$.
In addition, since $\gamma_{n,1}(K_\xi \setminus \partial K_\xi ) \subset \rs \setminus U$ for any $n \in \nn$, we have $K_\xi \setminus \partial K_\xi \subset F_\xi$  by Montel's theorem.
Therefore, $J_\xi \subset \partial K_\xi$.
\end{proof}

\begin{definition}\label{def-smallestFilledJ}
Let $S= (V, E,(\Gamma_e)_{e \in E})$  be a polynomial GDMS such that $\Gamma_e$ is a compact subset of  $\poly$  for all $e \in E$.
We denote by $K_i(S)$ the set of all points $z \in \rs$ such that $H_i(S)(z)$ is bounded in $\mathbb{C}$.
We call $K_i(S)$ the smallest filled-in Julia set of $S$ at $i \in V$.
\end{definition}

For the rest of this subsection, we consider a polynomial MRDS induced by $\tau$ and  $S_\tau = (V, E,(\Gamma_e)_{e \in E})$   such that $\Gamma_e$ is a compact subset of  $\poly$  for each $e \in E$.

\begin{lemma}\label{locconstonF_lemma}
The function $\mathbb{T}_{\infty, \tau}$ is locally constant on  $\mathbb{F}(S_\tau)$.
If $\tau$ is irreducible (i.e., $S_{\tau}$ is irreducible), then ${T}_{\infty, \tau}$ is  locally constant on ${F}(S_\tau)$.
\end{lemma}

\begin{proof}
Fix any connected component $U$ of the Fatou set $F_i(S_\tau)$ at $i \in V$.
For each $\xi = (\gamma_n , e_n )_{n \in \nn} \in X_i(S_\tau) $, it follows by Lemma \ref{inv_lemma} that  $\gamma_{n,1}(U) $ is contained in some connected component of $F_{t(e_n)}(S_\tau)$.
Thus,  for any point $z \in U$, we have $\gamma_{n,1}(z) \to \infty$ if and only if there exists $N \in \nn$ such that $\gamma_{N,1}(z)$ is contained in the  connected component of  $F_{t(e_N)}(S_\tau)$ which contains $\infty$.
Consequently, the function $\Bbb{T}_{\infty, \tau }(\cdot, i)$ is constant on $U$ and hence
$\Bbb{T}_{\infty ,\tau }$ is locally constant on $\Bbb{F}(S_{\tau })$.
If $S_{\tau }$ is irreducible, then $T_{\infty ,\tau }$ is locally constant on $F(S_{\tau })=\bigcap_{i \in V} F_i(S_\tau)$.
\end{proof}

\begin{lemma}\label{lem-filledEquiv}
\begin{enumerate}
\item $K_i(S_\tau) = \{ z \in \rs \, ;\, \mathbb{T}_{\infty,\tau }(z , i)=0 \} $ for all  $i \in V$.
\item The smallest filled-in Julia set  $K_i(S_\tau)$  is empty for all  $i \in V$ if and only if $\mathbb{T}_{\infty,\tau}(\cdot, i) \equiv 1$ for all $i \in V$.
\item If $\mathbb{T}_{\infty,\tau} \equiv 1$, then $\Bbb{J}_{\rm ker} (S_\tau) = \emptyset$.
\end{enumerate}
\end{lemma}

\begin{proof}
Evidentally, we have  $\Bbb{T}_{\infty ,\tau }(\cdot,i) \equiv 0$ on $K_{i}(S_{\tau })$.
Let $U_{\infty, j}$ be the connected component of  $F_j(S_\tau)$ which contains $\infty$.
Then the family $( U_{\infty, j} )_{j \in V}$ is forward $S_\tau$-invariant.
For any  $z \notin K_i(S_\tau) $, there exists  $h \in H_i^j(S_\tau)$ such that $h(z) \in U_{\infty, j}$.
Thus, there exist a finite admissible word $(e_1, \dots, e_N) \in E^N$ with initial vertex $i$ and maps  $\alpha_n \in \Gamma_{e_n} = \supp \tau_{e_n}$ such that $h=\alpha_N \circ \cdots \circ \alpha_1$.
For each $n=1, \dots,N$, there exists a neighborhood $A_n$ of $\alpha_n$ in $\poly$ such that $\gamma_N \circ \cdots \circ \gamma_1 (z) \in U_{\infty, j}$ for all  $\gamma_n \in A_n\  (n =1, \dots, N)$.
Now we set 
\[ \tilde{A} =A_1' \times \cdots \times A_N ' \times \prod_{N+1}^{\infty} (\poly \times E) , \quad A_n '=A_n \times \{ e_n \},\]
then $ \mathbb{T}_{\infty,\tau} (z , i) \geq \tilde{\tau}_i (\tilde{A}) > 0$.
This implies (i).
The rest of claims  easily follows from  (i) and Lemma \ref{lem-M} (with $U_{j}=U_{\infty ,j}$ for all $j\in V$).
\end{proof}

If $\# V =1$ (when the system is i.i.d.), then either ${T}_{\infty, \tau} \equiv 1$ or there exists  $z_{0} \in \mathbb{C}$ such that ${T}_{\infty, \tau}(z_{0})=0$ by Lemma \ref{lem-filledEquiv}.
However, this is not the case when $\# V > 1$ as the following Proposition \ref{prop-mrkv} shows.
This fact illustrates the difference between i.i.d. random dynamical systems and Markov case.
In other words, we found a new phenomenon which cannot hold in i.i.d. case.
For a concrete example of this phenomenon, see Example \ref{ex-mrkvT} .

\begin{prop}\label{prop-mrkv}
Suppose  ${\tau}$ is irreducible. 
If there exist $i,\,j \in V$ such that  $K_{i}(S_{\tau}) \neq \emptyset$ and $K_{i}(S_{\tau}) \cap K_{j}(S_{\tau})  = \emptyset $,   then ${T}_{\infty, \tau} \not \equiv 1$ and $T_{\infty, \tau}(z) >0$ for all $z \in \rs.$ 
\end{prop}

\begin{proof}
If $z\not\in K_{i}(S_{\tau })$, then $T_{\infty,\tau }(z)\geq p_{i}\Bbb{T}_{\infty ,\tau} (z,i)>0$ since $p_{i}>0.$ 
If $z\in K_{i}(S_{\tau })$, then $T_{\infty ,\tau }(z) \leq   \sum_{j \neq i}p_{j} < 1.$ 
Also, we have $z \not\in K_{j}(S_{\tau })$, and it follows that $T_{\infty ,\tau }(z)\geq p_{j}\Bbb{T}_{\infty ,\tau }(z,j)>0$ since $p_{j}>0$ 
\end{proof}

The following proposition claims that  $\mathbb{T}_{\infty , \tau}$ is continuous on $\mathbb{Y}$ if $\mathbb{J}_{{\rm {\rm ker}}} (S_\tau) = \emptyset $.
Combining this proposition with Corollary \ref{cor-kernelJuliaCond} or Lemma \ref{sepcond_lemma}, we obtain some examples of $\tau$ satisfying that  $\mathbb{T}_{\infty , \tau}$ is continuous.

\begin{prop}\label{prop-tInfty}
Let  $\phi \in {\rm C}(\mathbb{Y})$ be a continuous function with $\phi(\infty ,i)=1$ and $ \| \phi \|_{\mathbb{Y}} =1$.
Suppose that the support of $\phi(\cdot ,i)$ is contained in the connected component of $F_i(S_\tau)$ which contains $\infty$ for all $i \in V$.
Then the following statements hold.
\begin{enumerate}
\item The sequence $\{ M_\tau^n \phi \}_{n \in \nn}$ converges pointwise to $ \mathbb{T}_{\infty, \tau} $ on $\mathbb{Y}$ as $n \to \infty$.
\item The equation $M_\tau \mathbb{T}_{\infty, \tau}= \mathbb{T}_{\infty, \tau}$ holds.
\item If $\mathbb{J}_{{\rm {\rm ker}}} (S_\tau) = \emptyset $, then  $\{ M_{\tau }^{n} \phi \} _{n\in
\Bbb{N}}$ converges uniformly to $\Bbb{T}_{\infty ,\tau }$ on $\Bbb{Y}$ as $n\rightarrow
\infty $ and $\Bbb{T}_{\infty ,\tau }$ is continuous on $\Bbb{Y}.$ If, in addition to the assumption above,  ${\tau}$ is irreducible, then $T_{\infty ,\tau }$ is continuous on $\rs$.
\end{enumerate}
\end{prop}

\begin{proof}
Note that all sequences $\xi = (\gamma_n , e_n )_{n \in \nn} \in X_i(S_\tau)$ converge to $\infty$ locally uniformly on the connected component of $F_i(S_\tau)$ which contains $\infty$ for all $i \in V$.
\begin{enumerate}
\item Fix any point $z$  of $\rs$ and any $\xi = (\gamma_n , e_n )_{n \in \nn} \in X_i(S_\tau)$.
If $z \in A_\xi$, then $\phi( \xi_{n,1}(z, i)) \to 1$; othewise  $\phi( \xi_{n,1}(z, i)) = 0$ for all $n \in \nn$.
Thus, combining Lemma \ref{transop_lemma}  with  the dominated convergence theorem, we have
\begin{align*}
\lim_{n \to \infty} (M_{\tau} ^n \phi) (z,i) 
&= \lim_{n \to \infty}  \int_{X_i(S_\tau)} \phi (\xi_{n,1} (z,i)) \,  {\rm d} \tilde{\tau}_{i} (\xi)  \\
&= \int_{X_i (S_\tau)} \1_{ \{ \eta ; z \in A_\eta \} } (\xi) \,  {\rm d} \tilde{\tau}_{i} (\xi) =\mathbb{T}_{\infty ,\tau}(z,i).
\end{align*}
\item It follows immediately from (i).
\item If $\mathbb{J}_{{\rm {\rm ker}}}(S_\tau) = \emptyset $, then the sequence $\{ M_\tau^n \phi \}_{n \in \nn}$ is equicontinuous on $\mathbb{Y}$ by Proposition \ref{pathwisenull_prop}, Lemma \ref{FMiff_lemma} and Lemma \ref{fpt0_lemma}.
Moreover,  $\{ M_\tau^{n} \phi \}$ is uniformly bounded since $\| M_\tau^{n} \phi \|_{\mathbb{Y}} \leq \| \phi \|_{\mathbb{Y}} =1$.
By the Arzel\'a-Ascoli theorem, it follows that any subsequence of   $\{ M_\tau^{n} \phi \}$ has a subsequence which converges uniformly to  $\mathbb{T}_{\infty ,\tau}$ on $\mathbb{Y}$.
Therefore $\{ M_{\tau }^{n}\phi \} _{n\in \nn }$ converges  uniformly to $\Bbb{T}_{\infty ,\tau }$ on $\Bbb{Y}$ and  the limit  $\mathbb{T}_{\infty ,\tau}$ is continuous on $ \mathbb{Y}$.
\end{enumerate}
\end{proof}

The following example illustrates a new phenomenon which 
cannot hold in i.i.d.\ systems. For this example, we can apply
Proposition \ref{prop-mrkv} and Proposition \ref{prop-tInfty}.

\begin{example}\label{ex-mrkvT}
Let  $g_1(z) =z^2 -1,g_2(z)=z^2/4$ and set 
\begin{align*}
&f_{1}(z) = g_{1}\circ g_{1}(z-5) + 5,\, f_{2}(z) = g_{2}\circ g_{2}(z-5) +5, \\
&f_{3}(z) = g_{2}\circ g_{2}(z+5) -5,\, f_{4}(z) = g_{1}\circ g_{1}(z+5) - 5, \\
&h_{1}(z) = f_{1}(z+10),\, h_{2}(z) = f_{3}(z-10).
\end{align*}

We consider the polynomial MRDS  induced by 
$$\tau = \left(
    \begin{array}{cccc}
      \tau_{11} & \tau_{12} & \tau_{13} & \tau_{14} \\
      \tau_{21} & \tau_{22} & \tau_{23} & \tau_{24} \\
      \tau_{31} & \tau_{32} & \tau_{33} & \tau_{34} \\
      \tau_{41} & \tau_{42} & \tau_{43} & \tau_{44} \\
          \end{array}
  \right) = 
 \left(
    \begin{array}{cccc}
      \frac{1}{2} \delta_{f_{1}}&  \frac{1}{2}\delta_{f_{1}} &                                          & \\
      \frac{1}{4} \delta_{f_{2}} &  \frac{1}{4} \delta_{f_{2}} & \frac{1}{2}\delta_{h_{2}} & \\  
                                                &                                    &  \frac{1}{2} \delta_{f_{3}}&  \frac{1}{2}\delta_{f_{3}}  \\
       \frac{1}{2}\delta_{h_{1}} &                                        & \frac{1}{4} \delta_{f_{4}} &  \frac{1}{4} \delta_{f_{4}}   \\  
    \end{array}       \right).$$

See Figure \ref{fig-mrkvGDMS}. 
This system satisfies  the assumptions of Proposition \ref{prop-mrkv} and of Proposition \ref{prop-tInfty} (iii): 
$K_{1}(S_{\tau}) \neq \emptyset$ and $K_{1}(S_{\tau}) \cap K_{3}(S_{\tau})  = \emptyset $; 
$\mathbb{J}_{{\rm {\rm ker}}}(S_\tau) = \emptyset $ and $\tau$ is irreducible. 
Therefore, it follows that $T_{\infty, \tau} \not \equiv 1$,  $T_{\infty, \tau}(z) >0$ for all $z \in \rs$ and $T_{\infty, \tau}$ is continuous on $\rs.$ 
Figure \ref{fig-mrkvT} illustrates the  function $1 - T_{\infty, \tau}$ which represents the probability of not tending to infinity.

Moreover, the system in this  example is postcritically bounded; i.e. the set 
$$\overline{\bigcup_{h \in H(S_{\tau}) } \{ c \in \mathbb{C} ; c \text{ is a critical value of } h \}}\setminus \{\infty\}$$
is bounded in $\mathbb{C}.$
If the i.i.d. random dynamical system is  postcritically bounded,  then the connected components of Julia set ``surround'' one another and $T_{\infty, \tau}$ has the monotonicity with the surrounding order (see \cite[Theorem 2.4]{sumi15}). 
However, as this example illustrates, the ``surrounding order'' 
is not totally ordered regarding the set of connected components of
the Julia set of $S_{\tau }$ and $T_{\infty ,\tau }$ does not have the
monotonicity in a general non-i.i.d. irreducible system, even if the system is postcritically bounded.

    \begin{figure}[htbt]
    \begin{center}
\includegraphics[width=5cm]{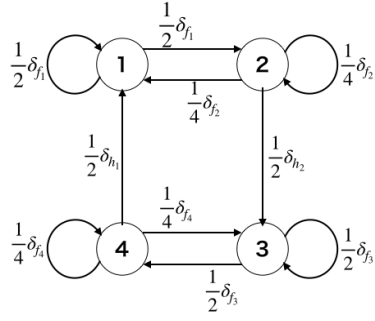}
\end{center}
 \caption{The schematic GDMS of Example \ref{ex-mrkvT}.}
\label{fig-mrkvGDMS}
\end{figure}

    \begin{figure}[htbt]
\begin{center}
\includegraphics[width=8cm]{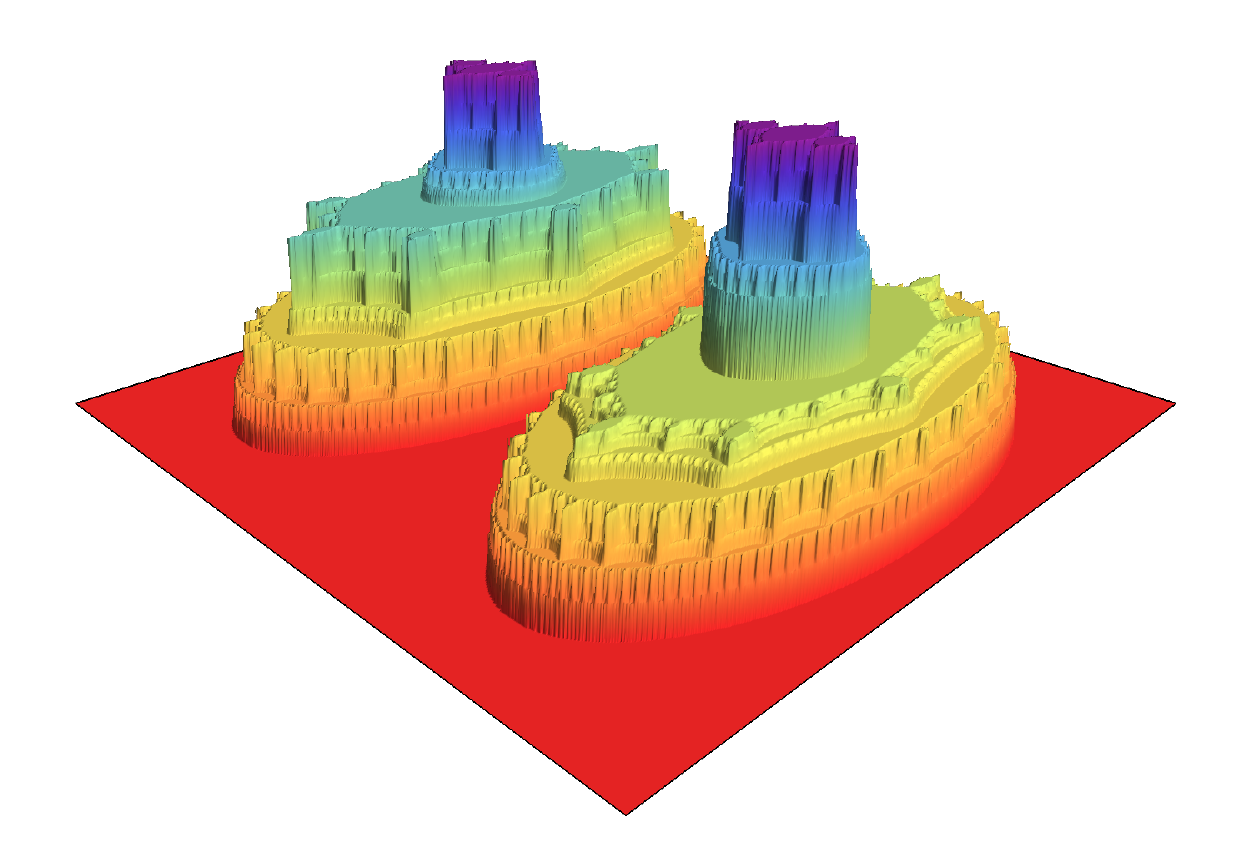}
\end{center}
 \caption{The graph of the function $1 - T_{\infty, \tau}$ with $0< T_{\infty, \tau}  \not \equiv 1$, which cannot hold in i.i.d. random dynamical systems of polynomials.}
\label{fig-mrkvT}
\end{figure}
\end{example}

\begin{cor}\label{cor-addNoise}
Suppose that the  polynomial GDMS $S_\tau$ satisfies the assumption (i) or (ii) of Corollary \ref{cor-kernelJuliaCond}.
Then $\Bbb{T}_{\infty ,\tau }$ is continuous on $\mathbb{Y}$ and $T_{\infty, \tau}$ is continuous on $\rs$.
\end{cor}

\begin{cor}\label{cor-conti}
Suppose that the polynomial GDMS $S_\tau$ satisfies the assumption of Lemma \ref{sepcond_lemma}.
Then $\Bbb{T}_{\infty ,\tau }$ is continuous on $\mathbb{Y}$.
Moreover, if $S_\tau$ is irreducible in addition to the assumption above, then $T_{\infty, \tau}$ is continuous on $\rs$.
\end{cor}

Corollary \ref{cor-addNoise} can be paraphrased by saying that  $\Bbb{T}_{\infty ,\tau }$ is continuous on $\mathbb{Y}$ if some $\Gamma_e= \supp \tau_{e}$ contains  sufficiently many polynomials.
In contrast, the backward separating condition (one of the assumptions of  Lemma \ref{sepcond_lemma}) seems to be familiar to the GDMS with less polynomials.
We focus on the latter case and show more sophisticated results.
We now begin with the following easy lemma.

\begin{lemma}\label{cor-NoException}
Suppose that  $\tau$ is irreducible and satisfies the backward separating condition.
Then $J_i (S_\tau) \cap \mathcal{E}_i (S_\tau) = \emptyset$ for all  $i \in V$. 
\end{lemma}

\begin{proof}
We divide the proof into two cases.
\begin{enumerate}
\item[Case 1.] Suppose that $S_\tau$ is essentially non-deterministic.
Then there exist two edges $e_1, e_2 \in E$ with the same initial vertex $j \in V$ and two maps $f_1 \in \Gamma_{e_1}, f_2 \in \Gamma_{e_2}$ such that either  $e_1\neq e_2$ or  $f_1 \neq f_2$.
Fix any $g_n \in H_{t(e_n)}^j (S_\tau)$ and set $h_n:=g_n \circ f_n \in H_j^j (S_\tau)$  for each $n =1,2$.
Then it is easy to see that $h_1^{-1}(J_j (S_\tau)) \cap h_2^{-1}(J_j (S_\tau))  = \emptyset$.
Now we have $h_n^{-1}(J_j (S_\tau) \cap \mathcal{E}_j (S_\tau) ) \subset J_j (S_\tau) \cap \mathcal{E}_j (S_\tau) $ for each $n=1,2$ and $\# (J_j (S_\tau) \cap \mathcal{E}_j (S_\tau)) \leq 2$ by Lemma \ref{exceptional_lemma}.
Therefore, we can show $J_j (S_\tau) \cap \mathcal{E}_j (S_\tau) = \emptyset $, and hence $J_i (S_\tau) \cap \mathcal{E}_i (S_\tau) = \emptyset$ for all  $i \in V$ since $S_\tau$ is irreducible.  
\item[Case 2.] Suppose that $\# \bigcup_{i(e)=j} \Gamma_e =1$ for all $j \in V$.
In this case, it follows that $H_i^i(S_\tau)$ is a polynomial semigroup generated by a single map $h_i$.
Then the Julia set $J_i(S_\tau)$ is equal to the Julia set $J(h_i)$ of $h_i$;
the exceptional set  $\mathcal{E}_i(S_\tau)$  is equal to the exceptional set  $\mathcal{E}(h_i)$ of  $h_i$.
Here,  $J(h_i)$ and $\mathcal{E}(h_i)$ is defined for the iteration of $h_i$, which is classically well known.
By \cite[Lemma 4.9]{miln06}, we have $J_i(S_\tau) \cap \mathcal{E}_i (S_\tau) = \emptyset$.
\end{enumerate}
\end{proof}

The following theorem is one of the main theorems in this paper 
 and
gives a sufficient condition that the function $\Bbb{T}_{\infty ,\tau }$ which
represents the probability of tending to infinity varies precisely on the Julia
set $\Bbb{J}(S_{\tau })$ and the function $\Bbb{T}_{\infty ,\tau }$ is
continuous on the whole space.  

\begin{theorem}\label{theorem-Tinfty}
Suppose that $\tau $ is irreducible and the polynomial GDMS $S_{\tau }$ satisfies the backward separating  condition and satisfies that  $\# \Gamma_e < \infty$ for all $e \in E$.
If  $K_j (S_\tau) \neq \emptyset$ for some $j \in V$, then the Julia set $J_i(S_\tau)$ at $i$ is equal to  the set of all points where $\mathbb{T}_{\infty, \tau} (\cdot, i)$  is not locally constant for all  $i \in V$.
Moreover, if, in addition to the assumption above, $S_\tau$ is essentially non-deterministic, then $\Bbb{T}_{\infty ,\tau }$ is continuous on $\mathbb{Y}$ and $\Bbb{T}_{\infty ,\tau } (J_i (S_\tau)\times \{ i \}) =[0,1]$ for all  $i \in V$, and hence $T_{\infty, \tau}$ is continuous on $\rs$.
\end{theorem}

\begin{proof}
First consider the case where $\# \bigcup_{i(e)=j} \Gamma_e =1$ for all $j \in V$.
Then it follows that $H_i^i(S_\tau)$ is a polynomial semigroup generated by a single map $h_i$, and hence
the smallest filled-in Julia set $K_i(S_\tau)$ is equal to the filled-in Julia set $K(h_i)$ of $h_i$, which is classically well known.
By definition, the function $\mathbb{T}_{\infty , \tau} (\cdot , i)$ is $0$ on  $K(h_i)$ and $1$ outside of $K(h_i)$.
Thus $\partial K(h_i) =J(h_i)=J_i(S_\tau)$ is equal to  the set of all points where $\mathbb{T}_{\infty, \tau}$  is not locally constant.
For the classical iteration theory, see \cite[\S 9]{miln06}.
(Remark: We denote by  $K(h)$ the set of all  $z \in \mathbb{C}$ for  which the orbit of $z$ under $h$ is bounded.
This set is called {\it filled} Julia set in \cite{miln06}.)

We next consider the case where there exist two edges  $e_1, e_2 \in E$ with the same initial vertex and two maps $f_1 \in \Gamma_{e_1}, f_2 \in \Gamma_{e_2}$ such that either  $e_1\neq e_2$ or  $f_1 \neq f_2$.
The proof is by contradiction.
Let $i \in V$ and 
suppose that  $\mathbb{T}_{\infty , \tau} (\cdot , i)$ is constant on a neighborhood  $U_0$  of  $z_0 \in J_i (S_\tau)$ in $\rs$.
Fix any  $z \in J_i (S_\tau)$.
By Lemma \ref{cor-NoException} and Lemma \ref{lem-backwrdOrbit}, we have  $J_i (S_\tau) = \overline{(H_i^i (S_\tau))^{-1}(z)}$. Thus there exists  $z' \in U_0 \cap (H_i^i (S_\tau))^{-1}(z)$, and hence there exists  $h \in H_i^i (S_\tau)$ such that  $h(z')=z$.
This $h$ can be written as $h= \alpha_N \circ \cdots \circ \alpha_1$, where  $( \alpha_n ,e_n)_{n=1}^{N}$ is an admissible finite sequence.
Set $\tilde{A}:= \prod_{n=1}^N ( \{ \alpha_n \} \times \{ e_n \}  ) \times \prod_{N+1}^{\infty} (\poly \times E)$.
Since  $\# \Gamma_e < \infty$, it follows that $\tilde{\tau}_i(\tilde{A}) >0$.
Now, for all $ \xi = (\gamma_n , e_n )_{n \in \nn} \in X_i (S_\tau) \setminus \tilde{A} =\tilde{A}^c$, we have $\xi _{N,1}(z' , i) \in \mathbb{F}(S_\tau)$ by the backward separating condition.
By Proposition \ref{prop-tInfty} and Lemma \ref{transop_lemma},
\begin{align*}
\mathbb{T}_{\infty,\tau} (z',i)&= \int_{X_i (S_\tau)} \mathbb{T}_{\infty,\tau}( \xi _{N,1}(z' , i) )\, {\rm d} \tilde{\tau}_i \\
& = \int_{\tilde{A}} \mathbb{T}_{\infty,\tau}( \xi _{N,1}(z' , i) )\, {\rm d} \tilde{\tau}_i +\int_{\tilde{A}^c}  \mathbb{T}_{\infty,\tau}( \xi _{N,1}(z' , i) )\, {\rm d} \tilde{\tau}_i \\
&=  \mathbb{T}_{\infty,\tau} (z , i)  \tilde{\tau}_i( \tilde{A}) + \int_{\tilde{A}^c}\mathbb{T}_{\infty,\tau}( \xi _{N,1}(z' , i) )\, {\rm d} \tilde{\tau}_i.
\end{align*}

We take a small neighborhood  $U' \subset U_0$ of $z'$ in $\rs$ so that  $\xi(z', i)$ and  $\xi(z_1', i)$ are contained in the same connected component of the Fatou set $\mathbb{F} (S_\tau)$ for all $z_1' \in U'$ and for all admissible sequences $\xi \neq ( \alpha_n ,e_n)_{n=1}^{N}$ with lengh  $N$. 
Then $h(U')$ is a neighborhood of  $h(z') =z$.
Take any $z_1 \in h(U')$ and take $z_1' \in U'$ so that  $h(z_1') =z_1$.
By our assumption that  $\mathbb{T}_{\infty , \tau} (\cdot , i)$ is constant on $ U_0$,
we have $\mathbb{T}_{\infty,\tau} (z',i) =\mathbb{T}_\infty (z_1',i)$.
Since $\mathbb{T}_{\infty,\tau}(\cdot,i)$ is constant on each connected component of $F_i (S_\tau)$ by Lemma \ref{locconstonF_lemma}, it follows  that 
\begin{align*}
\mathbb{T}_{\infty,\tau} (z,i) 
&= (\tilde{\tau}_i( \tilde{A}) )^{-1} \left( \mathbb{T}_{\infty,\tau} (z',i) - \int_{\tilde{A}^c}\mathbb{T}_{\infty,\tau}( \xi _{N,1}(z' , i) )\, {\rm d} \tilde{\tau}_i \right) \\
&= (\tilde{\tau}_i( \tilde{A}) )^{-1} \left( \mathbb{T}_{\infty,\tau} (z_1',i) - \int_{\tilde{A}^c}\mathbb{T}_{\infty,\tau}( \xi _{N,1}(z_1' , i) )\, {\rm d} \tilde{\tau}_i \right) \\
&=\mathbb{T}_{\infty,\tau} (z_1,i).
\end{align*}
Namely, $\mathbb{T}_{\infty,\tau} (\cdot,i)$ is constant on the neighborhood  $h(U')$  of  $z \in J_i (S_\tau)$.
It follows that $\mathbb{T}_{\infty, \tau}( \cdot , i)$ is locally constant at any point of $J_i(S_\tau)$.
However,  combining this with Lemma \ref{locconstonF_lemma}, we obtain that $\Bbb{T}_{\infty, \tau }(\cdot , i)$ is constant on $\hat{\Bbb{C}}$, which contradicts Lemma \ref{lem-filledEquiv}, irreducibility of $\tau $ and  the assumption that $K_j (S_\tau) \neq \emptyset$.
Thus $J_i(S_\tau)$ is equal to  the set of all points where $\mathbb{T}_{\infty, \tau}$  is not locally constant.

Moreover, the function $\Bbb{T}_{\infty ,\tau }$ is continuous on $\mathbb{Y}$
and $T_{\infty, \tau}$ is continuous on $\rs$ by Corollary \ref{cor-conti}.
Since
$\Bbb{T}_{\infty ,\tau }|_{K_{i}(S_\tau) \times \{ i\} }=0$ and
$\Bbb{T}_{\infty ,\tau }|_{\{\infty \}\times \{ i\}}=1$,
it follows that $\Bbb{T}_{\infty ,\tau }(\hat{\Bbb{C}}\times \{ i\} )=[0,1]$.
Further, since $\Bbb{T}_{\infty ,\tau }$ is locally constant on $F_{i}(S_{\tau })\times \{ i \} $, thus it follows that
$\Bbb{T}_{\infty ,\tau }(J_{i}(S_\tau)\times \{ i\} )=[0,1]$.
\end{proof}

Now we apply the results  to the following  random dynamical systems.
This is an immediate consequence of  Theorem \ref{theorem-Tinfty}, Corollary \ref{cor-cpt} and Proposition \ref{prop-intEmpty}.

 \begin{cor}\label{cor-foeExample}
For a given  $m \in \nn$, given $f_1, \dots,f_m \in \poly$ and a given irreducible stochastic matrix $P=(p_{ij})_{i,j = 1, \dots,m}$, we define  $\tau_{ij}$ as the  measure $p_{ij} \delta_{f_i}$, where $\delta_{f_i}$ denotes the Dirac measure at $f_i$.
Suppose that the polynomial GDMS $S_\tau$ induced by  $\tau = (\tau_{ij})$ satisfies
$K_i (S_\tau) \neq \emptyset$ for some  $i\in V$ and  $J_i (S_\tau)\cap J_j(S_\tau) = \emptyset$ for all  $i,j \in V$ with $i\neq j$.
Then  ${\rm int} (J(S_{\tau }))=\emptyset $ and the Julia set $J(S_\tau)$ is equal to  the set of all points where ${T}_{\infty, \tau}$  is not locally constant.
If, in addition to the assumption above, there exist $i,j,k \in \{1, \dots , m \}$ with $j\neq k$ such that $p_{ij} >0$ and $p_{ik} >0$ , 
then $T_{\infty,\tau}$ is continuous on $\rs$.
\end{cor}

\begin{example}\label{ex-main}
Let  $g_1(z) =z^2 -1,g_2(z)=z^2/4$ and set 
  $$m=2, \, f_i =g_i \circ g_i \, (i=1,2) \text{ and }  P =
\left(
    \begin{array}{cc}
      p_{11} & p_{12}  \\
      p_{21} & p_{22} \\
    \end{array}
  \right) = 
 \left(
    \begin{array}{cc}
      \frac{1}{2} &  \frac{1}{2} \\
      1 & 0 \\  
    \end{array}       \right).$$
Define $\tau_{ij}=p_{ij} \delta_{f_i}$ and $\tau = (\tau_{ij})$.
The polynomial GDMS  $S_\tau$ satisfies all the assumption of Corollary \ref{cor-foeExample}.

Figure \ref{fig-side} illustrates the grapf of $1-T_{\infty,\tau}$, which represents the probability of NOT tending to $\infty$.
The function $1-T_{\infty,\tau}$ is continuous on $\rs$ and varies on the Julia set  $J(S_\tau)$. 
Figure \ref{fig-top} illustrates the image of Figure \ref{fig-side} viewed from the top.
The Julia set  $J(S_\tau)$ is illustrated as the set of all points where the color varies.
\begin{figure}[h]
\center
\includegraphics[width=8cm]{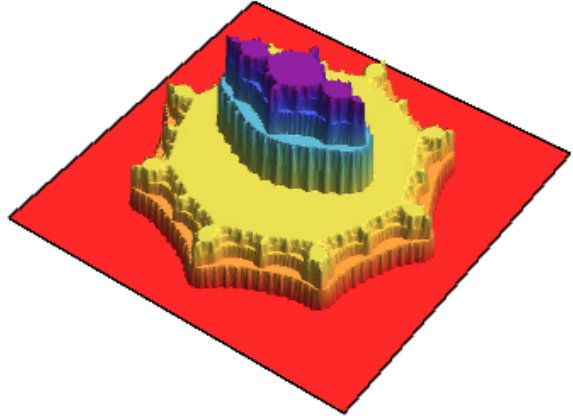}
 \caption{The graph of $1-T_{\infty, \tau}$ in Example \ref{ex-main}.
}
\label{fig-side}
\end{figure}

Figure \ref{ex-Julia} illustrates the Julia set $J (S_\tau)$.
The Julia set  contains neither  isolated points nor  interior points.

\begin{figure}[htbp]
\begin{minipage}{80truemm}
\begin{center}
\includegraphics[width=6cm, angle=90]{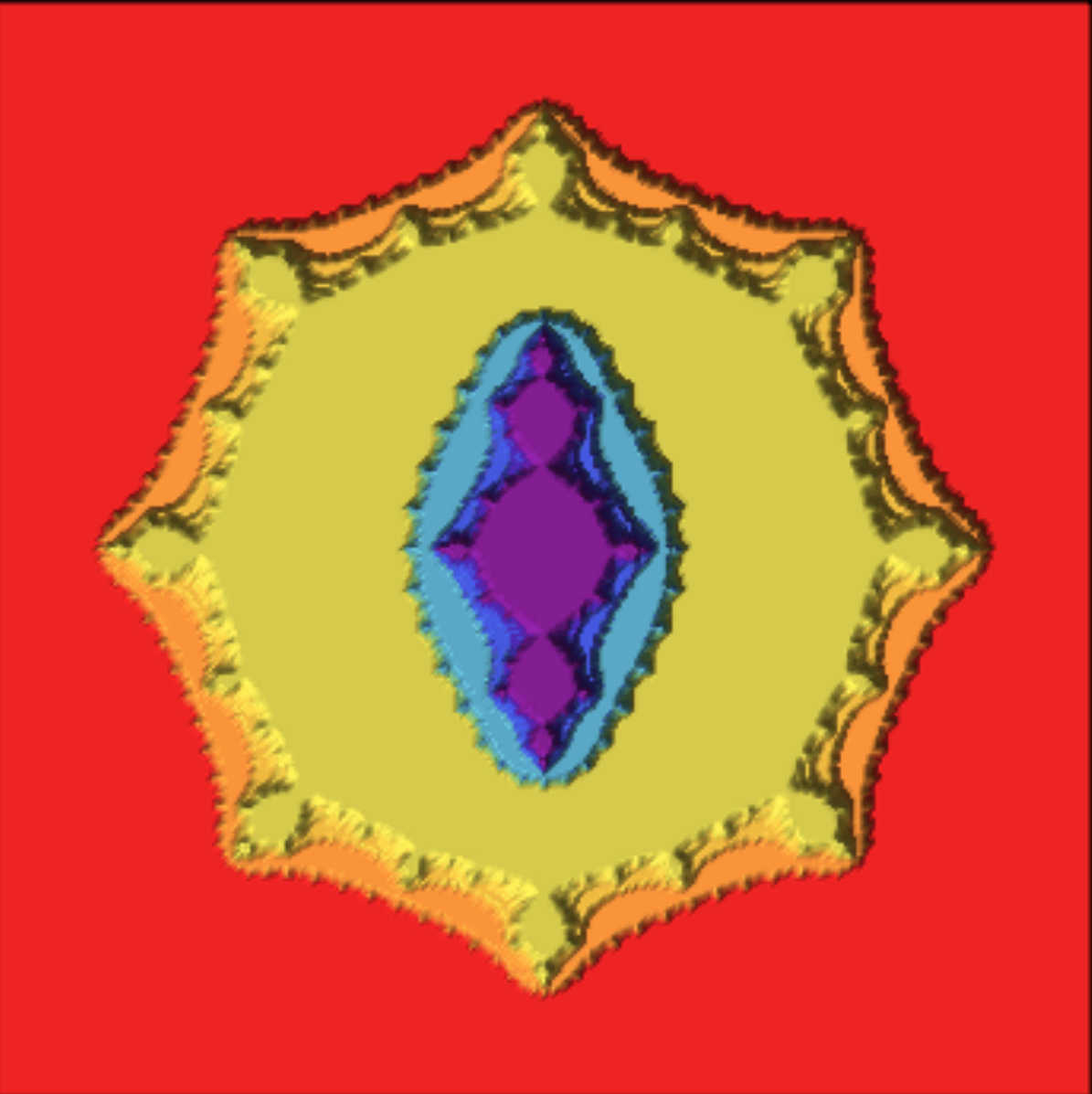}
\end{center}
 \caption{The graph of $1-T_{\infty, \tau}$ in Example \ref{ex-main}  viewed from above.
}
\label{fig-top}
\end{minipage}
\begin{minipage}{80truemm}
  \begin{center}
 \includegraphics[width=5cm]{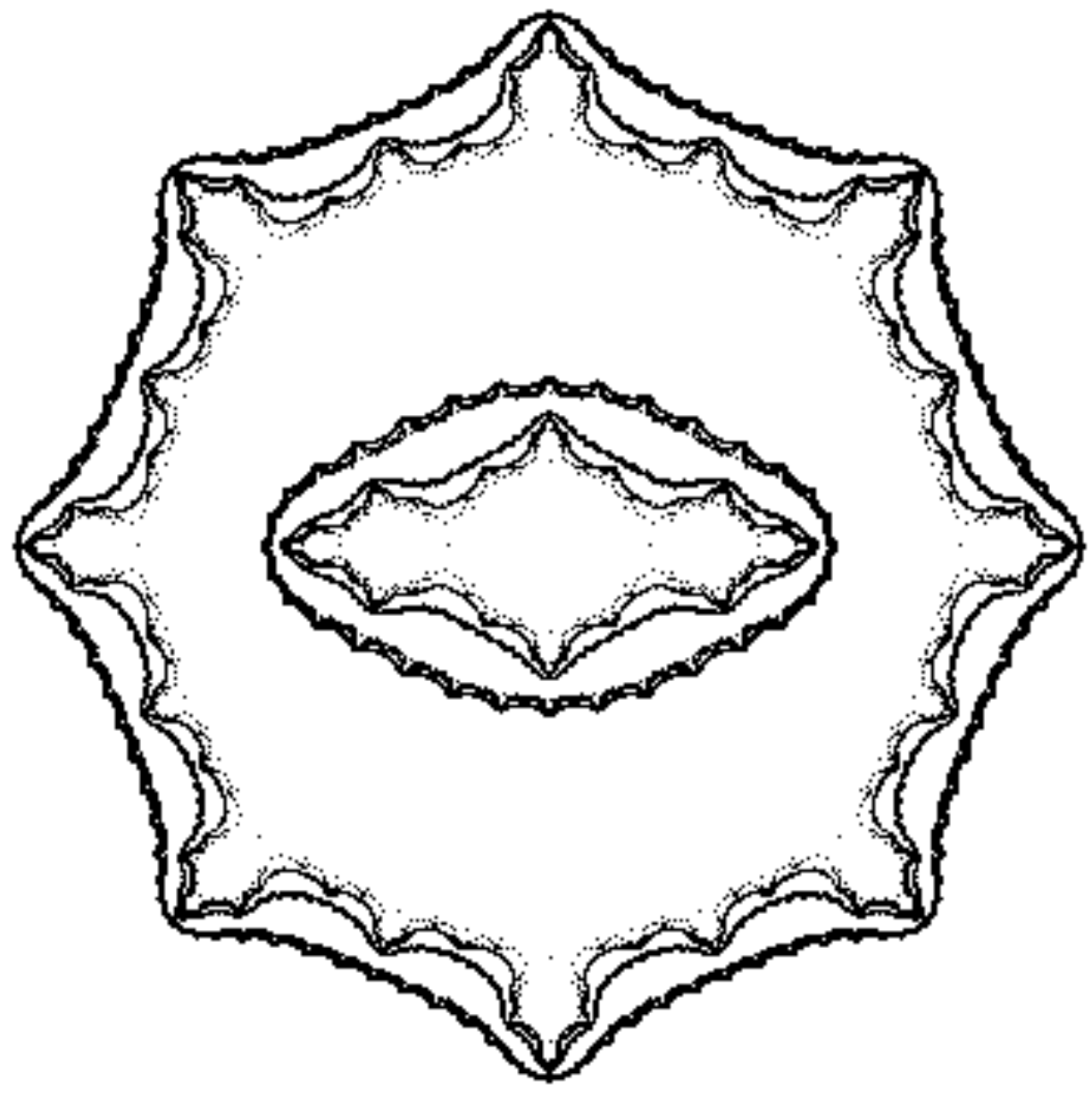}
 \end{center}
 \caption{The Julia set $J(S_\tau)$ for Example \ref{ex-main}.}
   \label{ex-Julia}
  \end{minipage}
\end{figure}

It follows from \cite[Example 6.2]{sumi11} that the Hausdorff dimension of the Julia set $J(S_\tau)$ is strictly less than $2$ for this example.
\end{example}

\vspace{30mm}

\end{document}